\documentclass[11pt]{article}

\usepackage[pagebackref,colorlinks=true,pdfpagemode=none,urlcolor=blue,
linkcolor=blue,citecolor=blue]{hyperref}

\usepackage{amsmath,amsfonts,amssymb,amsthm}
\usepackage{mathrsfs}
\usepackage{color}

\newtheorem{theorem}{Theorem}[section]
\newtheorem{lemma}[theorem]{Lemma}

\newtheorem{corollary}[theorem]{Corollary}
\newtheorem{proposition}[theorem]{Proposition}

\newtheorem{assumption}[theorem]{Assumption}
\theoremstyle{remark}
\newtheorem{remark}[theorem]{Remark}

\renewenvironment{proof}[1][Proof]{ {\itshape \noindent {#1.}} }{$\Box$
\medskip}

\numberwithin{equation}{section}
\newcommand{\R}{\mathbb{R}}

\newcommand{\Pb}{\mathbb{P}}
\newcommand{\E}{\mathbb{E}}
\newcommand{\F}{\mathcal{F}}

\newcommand{\G}{\mathcal{G}}

\newcommand{\C}{\mathcal{C}}

\newcommand{\V}{\mathbb{V}}

\newcommand{\eps}{\varepsilon}
\def\les{\lesssim}

\begin{document}

\title{Fluctuations of Parabolic Equations with Large Random Potentials}
\author{Yu Gu\thanks{Department of Applied Physics \& Applied
Mathematics, Columbia University, New York, NY 10027 (yg2254@columbia.edu; gb2030@columbia.edu)}  \and Guillaume Bal\footnotemark[1]}

\maketitle

\begin{abstract}
In this paper, we present a fluctuation analysis of a type of parabolic equations with large, highly oscillatory, random potentials around the homogenization limit. With a Feynman-Kac representation, the Kipnis-Varadhan's method, and a quantitative martingale central limit theorem, we derive the asymptotic distribution of the rescaled error between heterogeneous and homogenized solutions under different assumptions in dimension $d\geq 3$. The results depend highly on whether a stationary corrector exits.
\end{abstract}

\section{Introduction}

Equations with microscopic structure arise naturally in physics and applied science, and homogenization has become important to derive macroscopic models in both periodic and random settings, see \cite{kozlov1979averaging,papanicolaou1979boundary,yurinskii1986averaging,jikov1994homogenization}. When the underlying random medium is stationary and ergodic, stochastic homogenization replaces it by a deterministic, and properly-averaged constant, which, from a probabilistic point of view, is a law of large numbers type result. Much less is known regarding the random fluctuations though, e.g., the size of the error between heterogeneous and homogenized solutions, and the distribution of the rescaled error. The goal of this paper is to present a systematic analysis of random fluctuations produced by parabolic equations with large random potentials.

Error estimates have been derived for stochastic homogenization in different contexts, including the recent work on discrete and nonlinear setting \cite{yurinskii1986averaging,caffarelli2010rates,gloria2011optimal,gloria2012optimal,mourrat2012kantorovich}. However, asymptotic distributions are less well-understood. When the randomness is sufficiently mixing, it is natural to expect the central limit type of results to hold. For the homogenization constant, they are derived in \cite{nolen2011normal,biskup2012central}. For one dimensional case or equations with bounded random potentials, when certain integral representation of the solution is available, asymptotic distributions of the rescaled errors are derived for both short- and long-range-correlated randomness, leading to Gaussian or possible non-Gaussian limit \cite{figari1982mean,bourgeat1999estimates, bal2008central,bal2008random,bal2012corrector,gu2012random}.

In this paper, following the framework of \cite{gu2013weak}, we focus on the example of a parabolic equation with large, highly oscillatory, random potentials. A similar type of equations has been analyzed in \cite{bal2009convergence,bal2010homogenization,pardoux2006homogenization,pardoux2012homogenization,hairer2013random} to obtain either homogenization or convergence to stochastic partial differential equation (SPDE). Asymptotic Gaussian fluctuations are proved in \cite{bal2010homogenization} by combinatorial techniques for an equation with Gaussian potentials. One of the main goals here is to present an example of non-Gaussian potential for which such a result holds.

The main tool we use is a probabilistic representation and the Kipnis-Varadhan's method \cite{kipnis1986central}, which helps to reduce the error between heterogeneous and homogenized solutions to the Wasserstein distance between martingales and Brownian motions, plus residues caused by a corrector function. By a simple modification of the quantitative martingale central limit theorem developed by Mourrat \cite{mourrat2012kantorovich}, we obtain an accurate quantification of the Wasserstein distance, and are able to derive the asymptotic distribution under different assumptions in dimension $d\geq 3$. A similar approach will be applied to parabolic operators in divergence form in \cite{gu2014fluctuations}.

The results depend highly on the existence of a stationary corrector through the dimension. On one hand, when the stationary corrector does not exist in $d=3$, we prove a central limit result in Theorem \ref{thm:mainTH} for Gaussian and Poissonian potentials. The weak convergence limit can then be appropriately expressed as a stochastic parabolic equation with an additive noise. While the distribution we analyze is written as a conditional expectation by the probabilistic representation, we are able to link it to a parabolic equation with an additive random potential and eventually show that the random potential can be replaced by a white noise. On the other hand, when the stationary corrector exists in $d\geq 5$, for a large class of strongly mixing potentials, we show in Theorem \ref{thm:THd5} that the random fluctuation converges to the stationary corrector in distribution. The limit is not necessarily Gaussian, and the error decomposition there is consistent with a formal two-scale expansion.  For the critical dimension $d=4$ in which the stationary corrector does not exist, we present a decomposition of the error in Theorem \ref{thm:THd4} for equations with constant initial conditions.

The rest of the paper is organized as follows. We state the main results in Section \ref{sec:mainRe}. We then review some estimates obtained in \cite{gu2013weak} and prove a quantitative martingale central limit theorem in Section \ref{sec:quantMa}. In Section \ref{sec:proofMa}, \ref{sec:Inde} and \ref{sec:GaussLi}, we prove Theorem \ref{thm:mainTH} for $d=3$. In Section \ref{sec:THd4}, Theorem \ref{thm:THd5} and \ref{thm:THd4} are proved for $d\geq 5$ and $d=4$ respectively. Technical Lemmas are left in the Appendix.

Here are notations used throughout the paper. We use $\E$ to denote the expectation with respect to the random environment, and $\E_B,\E_W$ the expectations with respect to independent Browian motions $B_t,W_t$, respectively. We denote the normal distribution with mean $\mu$ and variance $\sigma^2$ by $N(\mu,\sigma^2)$, and $q_t(x)$ is the density function of $N(0,t)$. Let $G_\lambda(x)$ be the Green's function of $\lambda-\frac12\Delta$. Let $f^\lambda(x)=\int_{\R^d}\varphi(x-y)G_\lambda(y)dy, f_k^\lambda(x)=\int_{\R^d}\varphi(x-y)\partial_{x_k}G_\lambda(y)dy$, where $\varphi$ is the shape function of the Poissonian potential defined in Assumption \ref{ass:potential} below. The Fourier transform is denoted as $\F\{f\}(\xi)=\hat{f}(\xi)=\int_{\R^d}f(x)e^{-i\xi\cdot x}dx$. The convolution is denoted as $(f\star g)(x)=\int_{\R^d}f(x-y)g(y)dy$. When we write $a\les b$, it means $a\leq Cb$ for some $C>0$ independent of $\eps$. Let $a\wedge  b=\min(a,b)$, and $a\vee b=\max(a,b)$. For multidimensional integrations, $\prod_i dx_i$ is abbreviated as $dx$. Throughout the paper we assume the dimension $d\geq 3$.

\section{Problem setup and main results}
\label{sec:mainRe}

Let $(\Omega,\F,\Pb)$ be a random medium associated with a group of measure-preserving, ergodic transformations $\{\tau_x,x\in \R^d\}$, and $\E$ denote the expectation.  Let $\V\in L^2(\Omega)$ such that $\int_{\Omega}\V(\omega)\Pb(d\omega)=0$. Define the stationary random field $V(x,\omega)=\V(\tau_x \omega)$ and consider the following equation when $d\geq 3$:
\begin{equation}
\partial_t u_\eps(t,x,\omega)=\frac12\Delta u_\eps(t,x,\omega)+i\frac{1}{\eps}V(\frac{x}{\eps},\omega)u_\eps(t,x,\omega)
\label{eq:mainEq}
\end{equation}
with initial condition $u_\eps(0,x,\omega)=f(x)$ for $f\in \mathcal{C}_c^\infty(\R^d)$. We will omit the dependence on the particular realization $\omega$ and write $u_\eps(t,x)$ and $V(x)$ from now on.

Let $\{D_k,k=1,\ldots,d\}$ be the $L^2(\Omega)$ generator of $T_x$, which is defined as $T_xf(\omega)=f(\tau_x\omega)$, and the Laplacian operator $L=\frac12\sum_{k=1}^d D_k^2$. We use $\langle.,.\rangle$ to denote the inner product in $L^2(\Omega)$ and $\|.\|$ the $L^2(\Omega)$ norm. Assuming $T_x$ is strongly continuous in $L^2(\Omega)$, we obtain the spectral resolution
\begin{equation}
T_x=\int_{\R^d}e^{i\xi\cdot x}U(d\xi),
\end{equation}
where $U(d\xi)$ is the associated projection valued measure. We assume there is a non-negative power spectrum $\hat{R}(\xi)$ associated with $\V$, i.e., $\hat{R}(\xi)d\xi=(2\pi)^d\langle U(d\xi)\V,\V\rangle$. Clearly 
\begin{equation}
R(x)=\frac{1}{(2\pi)^{d}}\int_{\R^d}\hat{R}(\xi)e^{i\xi \cdot x}d\xi
\end{equation} is the covariance function of $V$. 

\cite[Theorem 2.2]{gu2013weak} shows that if $\hat{R}(\xi)|\xi|^{-2}$ is integrable, then$$u_\eps(t,x)\to u_{hom}(t,x)$$ in probability with $u_{hom}$ solving the homogenized equation
\begin{equation}
\partial_t u_{hom}(t,x)=\frac12\Delta u_{hom}(t,x)-\frac12\sigma^2 u_{hom}(t,x)
\label{eq:homoEq}
\end{equation}
with the same initial condition $u_{hom}(0,x)=f(x)$ and the homogenization constant $$\sigma^2=\frac{4}{(2\pi)^{d}}\int_{\R^d}\frac{\hat{R}(\xi)}{|\xi|^2}d\xi.$$ 
\begin{remark}
For the singularity $|\xi|^{-2}$ to be integrable around the origin, $d\geq 3$ is necessary.
\end{remark}

If an additional strongly mixing condition of $V$ is satisfied \cite[Assumption 2.4]{gu2013weak}, \cite[Theorem 2.6]{gu2013weak} proves an error estimate:
\begin{equation}
\E\{|u_\eps(t,x)-u_{hom}(t,x)|\}
\les\left\{
\begin{array}{ll}
\eps^\frac12 & d=3,\\
\eps |\log \eps|^\frac12 & d=4,\\
\eps & d\geq 5.
\end{array} \right.
\label{eq:SuErrEs}
\end{equation}

\begin{remark}
For the initial condition $f$, we actually only need the integrability of $\hat{f}(\xi)(1+|\xi|)$. If $f\equiv const$, since $\int_{\R^d}\delta(\xi)(1+|\xi|)d\xi=1$, heuristically we still have the integrability of $\hat{f}(\xi)(1+|\xi|)$. It can be checked that the estimate still holds.
\end{remark}

The goal of this paper is to go beyond the error estimate and analyze the rescaled fluctuation. In the following, we state the main results under different assumptions on the random potentials.

\subsection{Central limit theorem: $d=3$}

\begin{assumption}
$V$ is assumed to be Gaussian or Poissonian satisfying $\hat{R}(0)>0$, and
\begin{itemize}
\item when $V$ is Gaussian, for any $\alpha>0$, there exists $C_\alpha>0$ such that the covariance function satisfies $|R(x)|\leq C_\alpha( 1\wedge |x|^{-\alpha})$.
\item when $V$ is Poissonian, $V(x)=\int_{\R^d}\varphi(x-y)\omega(dy)-c_\varphi$, where the shape function $\varphi$ is continuous, compactly supported and satisfies $\int_{\R^d}\varphi(x)dx=c_\varphi$, and $\omega(dy)$ is the Poissonian point process with Lebesgue measure $dy$ as its intensity. Then $R(x)=\int_{\R^d}\varphi(x+y)\varphi(y)dy$ is compactly supported, and $\hat{R}(\xi)=|\hat{\varphi}(\xi)|^2$.
\end{itemize}
\label{ass:potential}
\end{assumption}

In particular, for the Poissonian case, $\hat{R}(0)>0$ implies $c_\varphi=\int_{\R^d}\varphi(x)dx\neq 0$, since $\hat{R}(0)=c_\varphi^2$.

The following is the main result.

\begin{theorem}[$d=3$]
Under Assumption \ref{ass:potential}, we have
\begin{equation}
\frac{u_\eps(t,x)-u_{hom}(t,x)}{\eps^\frac12}\Rightarrow v(t,x)
\end{equation}
weakly with $v(t,x)$ solving the following SPDE with additive spatial white noise and zero initial condition:
\begin{equation}
\partial_t v(t,x)=\frac12\Delta v(t,x)-\frac12\sigma^2 v(t,x)+i\sqrt{\hat{R}(0)}u_{hom}(t,x)\dot{W}(x).
\label{eq:mainTHSPDE}
\end{equation}
The weak convergence  is in the following sense:
\begin{enumerate}
\item As a process in $(t,x)\in \R_+\times \R^d$, the finite dimensional distributions of $\eps^{-\frac12}(u_\eps(t,x)-u_{hom}(t,x))\Rightarrow v(t,x)$ weakly.
\item The distribution of $\eps^{-\frac12}\int_{\R^d}(u_\eps(t,x)-u_{hom}(t,x))g(x)dx\Rightarrow \int_{\R^d}v(t,x)g(x)dx$ weakly for any fixed $t$ and test function $g\in \C_c^\infty(\R^d)$.
\end{enumerate}
\label{thm:mainTH}
\end{theorem}
It is clear that $v(t,x)$ is a Gaussian process, so Theorem \ref{thm:mainTH} can be regarded as a central limit result.

\subsection{Error decomposition by a corrector: $d\geq 4$}

For the Laplacian operater $L=\frac12\sum_{k=1}^dD_k^2$, a regularized corrector $\Phi_\lambda$ is defined by
\begin{equation}
(\lambda-L)\Phi_\lambda=\V
\end{equation}
for $\lambda>0$.
In Lemma \ref{lem:exSCor} below, we will show that the $L^2(\Omega)$ limit of $\Phi_\lambda$ exists iff $\hat{R}(\xi)|\xi|^{-4}$ is integrable. When the potential $V$ is short-range-correlated and $d\geq 5$, we can define the corrector $\Phi=\lim_{\lambda\to 0}\Phi_\lambda$ in $L^2(\Omega)$ and it is the solution of 
\begin{equation}
-L\Phi=\V.
\end{equation}

The following mixing assumption is the same as \cite[Assumption 2.4]{gu2013weak}.
\begin{assumption}[Strongly mixing assumption]
$\E\{V^6(x)\}<\infty$ and there exists a mixing coefficient $\rho(r)$ decreasing in $r\in [0,\infty)$ such that for any $\beta>0$, $\rho(r)\leq C_\beta (1\wedge r^{-\beta})$ for some $C_\beta>0$ and the following bound holds:
\begin{equation}
\E\{\phi_1(V)\phi_2(V)\}\leq \rho(r)\sqrt{\E\{\phi_1^2(V)\}\E\{\phi_2^2(V)\}}
\end{equation}
for any two compact sets $K_1,K_2$ with $d(K_1,K_2)=\inf_{x_1\in K_1,x_2\in K_2}\{|x_1-x_2|\}\geq r$ and any random variables $\phi_1(V),\phi_2(V)$ with $\phi_i(V)$ being $\mathcal{F}_{K_i}-$measurable and $\E\{\phi_i(V)\}=0$.
\label{ass:mixing}
\end{assumption}

\begin{theorem}[$d=4$]
Under Assumption \ref{ass:mixing}, if $f\equiv const$, we have for fixed $(t,x)$ that
\begin{equation}
u_\eps(t,x)=u_{hom}(t,x)+i\eps u_{hom}(t,x)\Phi_{\eps^2}(\tau_{\frac{x}{\eps}}\omega)+o(\eps|\log\eps|^\frac12),
\end{equation}
where $\frac{o(\eps|\log\eps|^\frac12)}{\eps|\log\eps|^\frac12}\to 0$ in $L^1(\Omega)$.
\label{thm:THd4}
\end{theorem}

We will see below that $\E\{|\eps \Phi_{\eps^2}(\tau_{\frac{x}{\eps}}\omega)|\}\les \eps|\log\eps|^\frac12$, so Theorem \ref{thm:THd4} implies for fixed $(t,x)$ that 
\begin{equation}
\frac{u_\eps(t,x)-u_{hom}(t,x)}{\eps|\log\eps|^\frac12}\sim iu_{hom}(t,x)\frac{\Phi_{\eps^2}(\tau_{\frac{x}{\eps}}\omega)}{|\log\eps|^\frac12}.
\end{equation}

It turns out that $|\log\eps|^{-\frac12}\Phi_{\eps^2}(\tau_{x}\omega)$ does not convergence in $L^2(\Omega)$, but we have the convergence in distribution.

\begin{corollary}
Under the assumption of Theorem \ref{thm:THd4}, if we further assume $V$ is Gaussian or Poissonian as in Assumption \ref{ass:potential}, then
\begin{equation}
\frac{\Phi_{\eps^2}(\tau_{x}\omega)}{|\log\eps|^{\frac12}}\Rightarrow N(0,\frac{4\hat{R}(0)}{(2\pi)^d})
\end{equation} in distribution, which implies
\begin{equation}
\frac{u_\eps(t,x)-u_{hom}(t,x)}{\eps|\log\eps|^\frac12}\Rightarrow iu_{hom}(t,x)N(0,\frac{4\hat{R}(0)}{(2\pi)^d})
\end{equation}
in distribution.
\label{cor:THd4}
\end{corollary}

\begin{theorem}[$d\geq 5$]
Under Assumption \ref{ass:mixing}, we have for fixed $(t,x)$ that
\begin{equation}
u_\eps(t,x)=u_{hom}(t,x)+i\eps u_{hom}(t,x)\Phi(\tau_{\frac{x}{\eps}}\omega)+i\eps u_{hom}(t,x)C_\V t+o(\eps),
\end{equation}
where $C_\V$ is some deterministic constant that can be computed explicitly, and $\frac{o(\eps)}{\eps}\to 0$ in $L^1(\Omega)$.
\label{thm:THd5}
\end{theorem}

$C_\V$ is given by \eqref{eq:biasCV}. If we assume some symmetry property of the distribution of $V(x)$, e.g., $\E\{V(x_1)V(x_2)V(x_3)\}=0,\forall x_1,x_2,x_3\in \R^d$ as in the Gaussian case, we have $C_\V=0$, i.e., the bias vanishes.

Since $\Phi(\tau_x\omega)$ is a stationary process, Theorem \ref{thm:THd5} implies that for fixed $(t,x)$, we have
\begin{equation}
\frac{u_\eps(t,x)-\E\{u_\eps(t,x)\}}{\eps}\Rightarrow iu_{hom}(t,x)\Phi(\tau_x\omega)
\end{equation}
in distribution as $\eps \to 0$. The limit is not necessarily Gaussian.

\subsection{Remarks on the results}

We first point out an important difference between the results in Theorem \ref{thm:mainTH} and Theorem \ref{thm:THd4}, \ref{thm:THd5}. When $d=3$, we obtain both the weak convergence for fixed $(t,x)$ and the weak convergence weakly in space. When $d\geq 4$, our approach only leads to the weak convergence for fixed $(t,x)$. Take $d\geq 5$ for example, Theorem \ref{thm:THd5} shows the random fluctuation  
\begin{equation}
u_\eps(t,x)-\E\{u_\eps(t,x)\}=i\eps u_{hom}(t,x)\Phi(\tau_{\frac{x}{\eps}}\omega)+o(\eps).
\label{eq:ranFluExd5}
\end{equation}
When considered weakly in space, it is actually much smaller than $\eps$. In general, for random variables of the form $\int_{\R^d}V(x/\eps)g(x)dx$ with $g\in \C_c^\infty$, we get an order of $\eps^{\frac{d}{2}}$. In our case, since the power spectrum of $\Phi(\tau_x\omega)$ blows up at the origin, we actually obtain $\int_{\R^d}i\eps u_{hom}(t,x)\Phi(\tau_{\frac{x}{\eps}}\omega)g(x)dx\sim \eps^{\frac{d-2}{2}}\ll \eps$. The size of the error is consistent with the result obtained by the second author for Gaussian potentials \cite[Theorem 2]{bal2010homogenization}, where it is shown that
 \begin{equation}
 \int_{\R^d}\frac{u_\eps(t,x)-\E\{u_\eps(t,x)\}}{\eps^{\frac{d-2}{2}}}g(x)dx \Rightarrow \int_{\R^d}v(t,x)g(x)dx
 \label{eq:ranFluBal}
 \end{equation}
in distribution. $v(t,x)$ in \eqref{eq:ranFluBal} is the formal solution to the SPDE \eqref{eq:mainTHSPDE} obtained in Theorem \ref{thm:mainTH} when $d=3$. Note that \eqref{eq:mainTHSPDE} is only well-posed when $d\leq 3$, but $\int_{\R^d}v(t,x)g(x)dx$ is well-defined in any dimension if we plug the formal Wiener integral expression of $v(t,x)$. However, it is straightforward to check that
\begin{equation}
\int_{\R^d}\frac{i\eps u_{hom}(t,x)\Phi(\tau_{\frac{x}{\eps}}\omega)}{\eps^{\frac{d-2}{2}}}g(x)dx\nRightarrow  \int_{\R^d}v(t,x)g(x)dx
\end{equation}
in distribution. On one hand, it indicates that \eqref{eq:ranFluExd5} only holds for fixed $(t,x)$ and is not true weakly in space, i.e., the $o(\eps)$ term actually contributes weakly in space. On the other hand, we note that Theorem \ref{thm:mainTH} is consistent with \eqref{eq:ranFluBal} when $d=3$.

Now we discuss the different assumptions we made on the random potentials. 

When $d=3$, we assume a Gaussian or Poissonian potential to obtain the following limiting SPDE after some explicit calculations:
\begin{equation*}
\partial_t v(t,x)=\frac12\Delta v(t,x)-\frac12\sigma^2 v(t,x)+i\sqrt{\hat{R}(0)}u_{hom}(t,x)\dot{W}(x).
\end{equation*}
From the above equation, the homogenization constant $\sigma^2$ shows up as a potential, and it comes from the averaging of $\eps^{-1}V(x/\eps)$. There is also the spatial white noise $\dot{W}(x)$ coming from the rescaled potential $\eps^{-\frac32}V(x/\eps)$. At a certain step, we need to get rid of the interaction between those two terms, and this is precisely the role of Proposition \ref{prop:asymInde}. Some explicit calculations facilitate our analysis.

For $d\geq 4$, we assume the strongly mixing property, also known as $\rho-$mixing, which is only used in an estimation of fourth-order moments. For the critical case $d=4$ with the logarithm scaling, we further assume the initial condition is constant to get rid of the interaction between $\eps\int_0^{t/\eps^2}V(B_s)ds$ and $\eps B_{t/\eps^2}$ appeared in the Feynman-Kac representation \eqref{eq:repreFK} below, and the martingale part does not contribute to the rescaled error in the end. Otherwise, we have the same term coming from the martingale part to deal with as $d=3$, see \eqref{eq:OmoreTd4}. Proving the central limit result when $d=4$ reduces to the weak convergence of $|\log\lambda|^{-\frac12}\int_{\R^d}G_\lambda(x-y)V(y)dy$, where $G_\lambda$ is the Green's function of $\lambda-\frac12\Delta$, and here we assume again a Gaussian or Poissonian potential.

In the end, we point out that the expansion obtained in Theorem \ref{thm:THd5} is consistent with a formal two-scale expansion. Let us assume that $u_\eps(t,x)=u_{hom}(t,x)+\eps u_1(t,x,y)+\ldots$ with a fast variable $y=x/\eps$, then by collecting terms of order $\eps^{-1}$ in \eqref{eq:mainEq}, we have the equation satisfied by $u_1$:
\begin{equation}
\frac{1}{\eps}\left(\frac12\Delta_y u_1(t,x,y)+iV(y)u_{hom}(t,x)\right)=0.
\end{equation}
The solution $u_1$ can be formally written as
\begin{equation}
u_1(t,x,y)=iu_{hom}(t,x)\int_{\R^d}G_0(y-z)V(z)dz,
\end{equation}
where $G_0$ is the Green's function of $-\frac12\Delta$. The integral is not defined realization-wise since $G_0$ is not integrable, but if we pass to the limit from the Green's function of $\lambda-\frac12\Delta$, we derive 
\begin{equation}
u_1(t,x,y)=\lim_{\lambda\to 0}iu_{hom}(t,x)\int_{\R^d}G_\lambda(y-z)V(z)dz=iu_{hom}(t,x)\Phi(\tau_{\frac{x}{\eps}}\omega),
\end{equation}
then the formal expansion gives $u_\eps(t,x)=u_{hom}(t,x)+i\eps u_{hom}(t,x)\Phi(\tau_{\frac{x}{\eps}}\omega)+\ldots$, which is consistent with Theorem \ref{thm:THd5}. This indicates that when a stationary corrector exists, it is possible to obtain the random fluctuation by a formal two-scale expansion.

\section{Refining the error}
\label{sec:quantMa}

In this section, we review some key estimates in \cite{gu2013weak}, prove a quantitative martingale central limit theorem, and derive a compact form of the properly-rescaled error in \eqref{eq:decom1Cor}, \eqref{eq:decom1Cord4} and \eqref{eq:decom1Cord5} for $d=3,4$ and $d\geq 5$ respectively.

\subsection{Error estimates}
By the Feynman-Kac representation and the scaling property of Brownian motion, the solution to \eqref{eq:mainEq} is written as
\begin{equation}
u_\eps(t,x)=\E_B\{f(x+\eps B_{t/\eps^2})\exp(i\eps\int_0^{t/\eps^2}V(\frac{x}{\eps}+B_s)ds)\}.
\label{eq:repreFK}
\end{equation}
Define $y_s:=\tau_{\frac{x}{\eps}+B_s}\omega$ as the environmental process taking values in $\Omega$, and the regularized corrector $\Phi_\lambda$ solve the corrector equation$
(\lambda-L)\Phi_\lambda=\V$.
We choose $\lambda=\eps^2$ from now on.
By It\^o's formula, the process $X_t^\eps:=\eps\int_0^{t/\eps^2}\V(y_s)ds$ can be decomposed as $X_t^\eps=R_t^\eps+M_t^\eps$ with
\begin{eqnarray}
R^\eps_t:&=&\eps\int_0^{t/\eps^2}\lambda\Phi_\lambda(y_s)ds-\eps\Phi_\lambda(y_{t/\eps^2})+\eps\Phi_\lambda(y_0),\\
M^\eps_t:&=&\eps\int_0^{t/\eps^2}\sum_{k=1}^d D_k\Phi_\lambda(y_s)dB^k_s.
\end{eqnarray}

By \cite[Proposition 3.1]{gu2013weak}, we have \begin{equation}
\E\E_B\{|R_t^\eps|^2\}\les \lambda\langle \Phi_\lambda,\Phi_\lambda\rangle\les \sqrt{\lambda}1_{d=3}+\lambda|\log\lambda|1_{d=4}+\lambda1_{d\geq 5}.
\label{eq:EstiR}
\end{equation}

If we define $\sigma_\lambda^2=\sum_{k=1}^d \|D_k\Phi_\lambda\|^2$, then by \cite[Proposition 3.2]{gu2013weak},
\begin{equation}
 |\sigma_\lambda^2-\sigma^2|\les \sqrt{\lambda}1_{d=3}+\lambda|\log\lambda|1_{d=4}+\lambda1_{d\geq 5}.
\end{equation}

The error is then decomposed into three parts, $u_\eps(t,x)-u_{hom}(t,x)=\mathcal{E}_1+\mathcal{E}_2+\mathcal{E}_3$ with
\begin{eqnarray}
\mathcal{E}_1&=&\E_B\{f(x+\eps B_{t/\eps^2})e^{iX_t^\eps}\}-\E_B\{f(x+\eps B_{t/\eps^2})e^{iM_t^\eps}\},\\
\mathcal{E}_2&=&\E_B\{f(x+\eps B_{t/\eps^2})e^{iM_t^\eps}\}-\E_B\{f(x+\eps B_{t/\eps^2})e^{-\frac12\sigma_\lambda^2t}\},\\
\mathcal{E}_3&=&\E_B\{f(x+\eps B_{t/\eps^2})e^{-\frac12\sigma_\lambda^2t}\}-\E_B\{f(x+\eps B_{t/\eps^2})e^{-\frac12\sigma^2t}\},
\end{eqnarray}
so we have \begin{eqnarray}
\E\{|\mathcal{E}_1|\}&\les& \E\E_B\{|R_t^\eps|\}\les \sqrt{\lambda\langle\Phi_\lambda,\Phi_\lambda\rangle}\les \eps^\frac121_{d=3}+\eps|\log\eps|^\frac121_{d=4}+\eps1_{d\geq 5},\\
|\mathcal{E}_3|&\les& |\sigma_\lambda^2-\sigma^2|\les \eps1_{d=3}+\eps^2|\log\eps|1_{d=4}+\eps^21_{d\geq 5}.
\end{eqnarray} 
Clearly, $\mathcal{E}_1$ is of the right order given by \eqref{eq:SuErrEs} and $\mathcal{E}_3\to 0$ after being properly rescaled. $\mathcal{E}_2$ is analyzed through a quantitative martingale central limit theorem. First it is written in the Fourier domain as
\begin{equation}
\mathcal{E}_2=\int_{\R^d} \frac{1}{(2\pi)^d}\hat{f}(\xi)e^{i\xi \cdot x}\E_B\{e^{i\eps \xi\cdot B_{t/\eps^2}+iM_t^\eps}-e^{-\frac12(|\xi|^2+\sigma_\lambda^2)t}\}d\xi.
\label{eq:iiFourier}
\end{equation}
Define $\tilde{M}_t^\eps:=\eps \xi\cdot B_{t/\eps^2}+M_t^\eps$, then \cite[Theorem 3.2]{mourrat2012kantorovich} implies
\begin{equation}
\E\{|\mathcal{E}_2|\}\les \int_{\R^d}|\hat{f}(\xi)|\E\E_B\{|\langle \tilde{M}^\eps\rangle_t-(|\xi|^2+\sigma_\lambda^2)t|\}d\xi.
\end{equation}
Since $\langle \tilde{M}^\eps\rangle_t-(|\xi|^2+\sigma_\lambda^2)t=\eps^2\int_0^{t/\eps^2}\left(\sum_{k=1}^d D_k\Phi_\lambda(y_s)^2-\sigma_\lambda^2\right)ds+2\eps^2\int_0^{t/\eps^2}\sum_{k=1}^d\xi_kD_k\Phi_\lambda(y_s)ds$,
by a second moment estimate in \cite[Proposition 3.5, Lemma 3.6]{gu2013weak}, we obtain
\begin{eqnarray}
\E\E_B\{|\eps^2\int_0^{t/\eps^2}\left(\sum_{k=1}^d D_k\Phi_\lambda(y_s)^2-\sigma_\lambda^2\right)ds|^2\}&\les& \eps^2|\log\eps|1_{d=3}+\eps^21_{d\geq 4},
\label{eq:EstiQuadra1}\\
\E\E_B\{|2\eps^2\int_0^{t/\eps^2}\sum_{k=1}^d\xi_kD_k\Phi_\lambda(y_s)ds|^2\}&\les& \eps1_{d=3}+\eps^2|\log\eps|1_{d=4}+\eps^21_{d\geq 5},
\label{eq:EstiQuadra}
\end{eqnarray}
so
\begin{equation}
\E\E_B\{|\langle \tilde{M}^\eps\rangle_t-(|\xi|^2+\sigma_\lambda^2)t|^2\}\les \eps1_{d=3}+\eps^2|\log\eps|1_{d=4}+\eps^21_{d\geq 5},
\label{eq:2ndMoQua}
\end{equation}
which implies 
\begin{equation}
\E\{|\mathcal{E}_2|\}\les\eps^\frac121_{d=3}+\eps|\log\eps|^{\frac12}1_{d=4}+\eps1_{d\geq 5}.
\end{equation} Therefore, to analyze the asymptotic distribution of $u_\eps(t,x)-u_{hom}(t,x)$ after proper rescaling, we need to refine $\mathcal{E}_1$ and $\mathcal{E}_2$ to separate those terms of the right order.

\begin{remark}
When applying a refined quantitative martingale central limit theorem to analyze $\mathcal{E}_2$, we will use \eqref{eq:2ndMoQua} frequently.
\end{remark}

\subsection{Quantitative martingale central limit theorem}

For $\mathcal{E}_1$, using the fact that $|e^{ix}-1-ix|\les |x|^2$, we have that
\begin{equation}
\begin{aligned}
&\E\{|\E_B\{f(x+\eps B_{t/\eps^2})e^{iX_t^\eps}\}-\E_B\{f(x+\eps B_{t/\eps^2})(1+iR_t^\eps)e^{iM_t^\eps}\}|\}\\
\les &\eps1_{d=3}+\eps^2|\log \eps|1_{d=4}+\eps^21_{d\geq 5}.
\end{aligned}
\end{equation}
so we have $\E\{|\mathcal{E}_1-v_{1,\eps}|\}\ll \eps^\frac121_{d=3}+\eps|\log\eps|^\frac121_{d=4}+\eps1_{d\geq 5}$ with
\begin{equation}
v_{1,\eps}:=\E_B\{f(x+\eps B_{t/\eps^2})iR_t^\eps e^{iM_t^\eps}\}.
\end{equation}

Now we analyze $\mathcal{E}_2$. By the expression in \eqref{eq:iiFourier}, the goal is reduced to an estimation of $\E_B\{e^{i\eps \xi\cdot B_{t/\eps^2}+iM_t^\eps}-e^{-\frac12(|\xi|^2+\sigma_\lambda^2)t}\}$ and separating the terms of the right order. The following is a simply modified quantitative martingale central limit theorem we need.
\begin{proposition}
Let $M_t$ be a continuous martingale with a right-continuous filtration $(\F_t)_{t\geq 0}$ and $W_t$ a standard Brownian motion, then for any $f\in \C_b(\R)$ with up to third order bounded and continuous derivatives, we have
\begin{equation}
|\E\{f(M_1)-f(W_1)-\frac12f''(M_\tau)(\langle M\rangle_1-1)\}|\leq C\E\{|\langle M\rangle_1-1|^{\frac32}\},
\end{equation}
where $\tau=\sup\{s\in [0,1]|\langle M\rangle_s\leq 1\}$ and the constant $C$ only depends on the bound of $f'''$.
\label{prop:3rdQuantCLT}
\end{proposition}

\begin{proof}
The proof follows a special case of \cite[Theorem 3.2]{mourrat2012kantorovich}.

Since $M_t$ is continuous, the quadratic variation process $\langle M\rangle_t$ is continuous as well. It is clear that $\tau$ is a stopping time, and we construct $\tilde{M}_t$ on $[0,2]$ as
\begin{equation}
\tilde{M}_t=
 \left\{
\begin{array}{ll}
M_t & t\in [0,\tau],\\
M_\tau& t\in (\tau,1],\\
M_\tau+b_{t-1} & t\in (1,2-\langle M\rangle_\tau],\\
M_\tau+b_{1-\langle M\rangle_\tau} & t\in (2-\langle M\rangle_\tau,2],
\end{array} \right.
\end{equation}
where $b$ is an independent Brownian motion starting from the origin with a right continuous filtration $(\F^b_t)_{t\geq 0}$.

Clearly $\tilde{M}_t$ is a continuous martingale with the new filtration $\tilde{\F}_t=\sigma(\F_t\cup \F^b_0)$ when $t\leq 1$ and $\tilde{\F}_t=\sigma(\F_t\cup \F^b_{t-1})$ when $t>1$. $\langle\tilde{M}\rangle_2=1$, so $\tilde{M}_2\sim N(0,1)$, which implies $\E\{f(M_1)-f(W_1)\}=\E\{f(M_1)-f(\tilde{M}_2)\}$. We write
\begin{equation}
f(M_1)-f(\tilde{M}_2)=f(M_1)-f(M_\tau)-(f(\tilde{M}_2)-f(M_\tau)).
\end{equation}

For the first term, we have
\begin{equation}
\begin{aligned}
&|\E \{f(M_1)-f(M_\tau)-(M_1-M_\tau)f'(M_\tau)-\frac12(M_1-M_\tau)^2f''(M_\tau)\}|\\
=&|\E \{f(M_1)-f(M_\tau)-\frac12(\langle M\rangle_1-\langle M\rangle_\tau) f''(M_\tau)\}|
\leq C\E\{|M_1-M_\tau|^3\}.
\end{aligned}
\end{equation}

For the second term, we have $\tilde{M}_2=M_\tau+b_{1-\langle M\rangle_\tau}$, so
\begin{equation}
\begin{aligned}
&|\E \{f(\tilde{M}_2)-f(M_\tau)- b_{1-\langle M\rangle_\tau}f'(M_\tau)-\frac12b_{1-\langle M\rangle_\tau}^2f''(M_\tau)\}|\\
=&|\E \{f(\tilde{M}_2)-f(M_\tau)-\frac12(1-\langle M\rangle_\tau)f''(M_\tau)\}|\leq C\E \{|b_{1-\langle M\rangle_\tau}|^3\}\leq C\E\{(1-\langle M\rangle_\tau)^{\frac32}\}.
\end{aligned}
\end{equation}

Note that $\E\{|M_1-M_\tau|^3\}\leq C\E\{(\langle M\rangle_1-\langle M\rangle_\tau)^\frac32\}\leq C\E\{|\langle M\rangle_1-1|^\frac32\}$ and the same estimate holds for $\E\{(1-\langle M\rangle_\tau)^{\frac32}\}$. The proof is complete.
\end{proof}

For almost every $\omega\in \Omega$, $\tilde{M}_t^\eps=\eps \xi\cdot B_{t/\eps^2}+M_t^\eps$ is a continuous, square-integrable martingale, we apply Proposition \ref{prop:3rdQuantCLT} with $f=e^{ix}$ and obtain for almost every $\omega$ that
\begin{equation}
\begin{aligned}
&|\E_B\{e^{i\tilde{M}_t^\eps}-e^{-\frac12(|\xi|^2+\sigma_\lambda^2)t}+\frac12e^{i\tilde{M}_\tau^\eps}(\langle \tilde{M}^\eps\rangle_t-(|\xi|^2+\sigma_\lambda^2)t)\}|\\
\les &\E_B\{|\langle \tilde{M}^\eps\rangle_t-(|\xi|^2+\sigma_\lambda^2)t|^{\frac32}\}
\end{aligned}
\end{equation}
where $\tau:=\sup\{s\in [0,t]|:\eps^2\int_0^{s/\eps^2}\sum_{k=1}^d (\xi_k+D_k\Phi_\lambda(y_s))^2ds\leq (|\xi|^2+\sigma_\lambda^2)t\}$.

First we have
\begin{equation}
\begin{aligned}
&\E\{|\mathcal{E}_2-\int_{\R^d} \frac{1}{(2\pi)^d}\hat{f}(\xi)e^{i\xi \cdot x}\E_B\{-\frac12e^{i\tilde{M}_\tau^\eps}(\langle \tilde{M}^\eps\rangle_t-(|\xi|^2+\sigma_\lambda^2)t)\}d\xi|\}\\
\les& \int_{\R^d}|\hat{f}(\xi)|\E\E_B\{|\langle \tilde{M}^\eps\rangle_t-(|\xi|^2+\sigma_\lambda^2)t|^{\frac32}\}d\xi\\
\les&\int_{\R^d}|\hat{f}(\xi)|\left(\E\E_B\{|\langle \tilde{M}^\eps\rangle_t-(|\xi|^2+\sigma_\lambda^2)t|^2\}\right)^{\frac34}d\xi\ll 
\eps^\frac121_{d=3}+\eps|\log\eps|^{\frac12}1_{d=4}+\eps1_{d\geq 5}
\end{aligned}
\end{equation}
by recalling \eqref{eq:2ndMoQua}.

Next, we consider
\begin{equation}
\begin{aligned}
&\E\{|\int_{\R^d} \frac{1}{(2\pi)^d}\hat{f}(\xi)e^{i\xi \cdot x}\E_B\{\frac12\left(e^{i\tilde{M}_\tau^\eps}-e^{i\tilde{M}_t^\eps}\right)(\langle \tilde{M}^\eps\rangle_t-(|\xi|^2+\sigma_\lambda^2)t)\}d\xi|\}\\
\leq &\int_{\R^d} \frac{1}{(2\pi)^d}|\hat{f}(\xi)|\frac12\sqrt{\E\E_B\{|\tilde{M}_\tau^\eps-\tilde{M}_t^\eps|^2\}}\sqrt{\E\E_B\{|\langle \tilde{M}^\eps\rangle_t-(|\xi|^2+\sigma_\lambda^2)t|^2\}}d\xi\\
\leq &\int_{\R^d} \frac{1}{(2\pi)^d}|\hat{f}(\xi)|\frac12\sqrt{\E\E_B\{\langle \tilde{M}^\eps\rangle_t-\langle \tilde{M}^\eps\rangle_\tau\}}\sqrt{\E\E_B\{|\langle \tilde{M}^\eps\rangle_t-(|\xi|^2+\sigma_\lambda^2)t|^2\}}d\xi\\
\leq &\int_{\R^d} \frac{1}{(2\pi)^d}|\hat{f}(\xi)|\frac12\sqrt{\E\E_B\{|\langle \tilde{M}^\eps\rangle_t-(|\xi|^2+\sigma_\lambda^2)t|\}}\sqrt{\E\E_B\{|\langle \tilde{M}^\eps\rangle_t-(|\xi|^2+\sigma_\lambda^2)t|^2\}}d\xi\\
\ll&\eps^\frac121_{d=3}+\eps|\log\eps|^{\frac12}1_{d=4}+\eps1_{d\geq 5}
\end{aligned}
\end{equation}
again by using \eqref{eq:2ndMoQua}.

In the end, since 
\begin{equation}
\langle\tilde{M}^\eps\rangle_t-(|\xi|^2+\sigma_\lambda^2)t=2\eps^2\int_0^{t/\eps^2}\sum_{k=1}^d \xi_k D_k\Phi_\lambda(y_s)ds+\eps^2\int_0^{t/\eps^2}(\sum_{k=1}^d D_k\Phi_\lambda(y_s)^2-\sigma_\lambda^2)ds,
\end{equation}
we obtain the following results by \eqref{eq:EstiQuadra1} and \eqref{eq:EstiQuadra}. For $\approx$, it means the difference goes to zero in $L^1(\Omega)$ as $\eps\to 0$.

When $d=3$, $\E\{|\mathcal{E}_2-v_{2,\eps}|\}\ll\eps^\frac12$, where
\begin{equation}
v_{2,\eps}=-\int_{\R^d} \frac{1}{(2\pi)^d}\hat{f}(\xi)e^{i\xi \cdot x}\E_B\{e^{i\tilde{M}_t^\eps}\eps^2\int_0^{t/\eps^2}\sum_{k=1}^d \xi_k D_k\Phi_\lambda(y_s)ds\}d\xi.
\label{eq:OmoreTd4}
\end{equation}
By writing $v_{1,\eps}$ in Fourier domain as well, we have proved that
\begin{equation}
\begin{aligned}
\frac{u_\eps(t,x)-u_{hom}(t,x)}{\eps^\frac12}\approx&\frac{v_{1,\eps}+v_{2,\eps}}{\eps^\frac12}\\
= &\int_{\R^d}\frac{1}{(2\pi)^d}\hat{f}(\xi)e^{i\xi \cdot x}\eps^{-\frac12}\E_B\{e^{i\tilde{M}_t^\eps}\left(iR_t^\eps-\eps^2\int_0^{t/\eps^2}\sum_{k=1}^d \xi_k D_k\Phi_\lambda(y_s)ds\right)\} d\xi.
\end{aligned}
\label{eq:decom1Cor}
\end{equation}

When $d=4$, if $f(x)\equiv const$, without loss of generality let $f(x)\equiv 1$, then $\hat{f}(\xi)=\delta(\xi)$, and in the Fourier domain the integration only charges $\xi=0$, so only the bound in \eqref{eq:EstiQuadra1} matters for $\mathcal{E}_2$ and we have  $\E\{|\mathcal{E}_2|\}\les \eps \ll \eps|\log \eps|^\frac12$. Therefore, we obtain
\begin{equation}
\frac{u_\eps(t,x)-u_{hom}(t,x)}{\eps|\log\eps|^\frac12}\approx \frac{v_{1,\eps}}{\eps|\log\eps|^\frac12}=\frac{\E_B\{iR_t^\eps e^{iM_t^\eps}\}}{\eps|\log\eps|^\frac12}
\label{eq:decom1Cord4}
\end{equation}

When $d\geq 5$, $\E\{|\mathcal{E}_2-v_{2,\eps}|\}\ll\eps$, where
\begin{equation}
v_{2,\eps}=-\frac12\int_{\R^d} \frac{1}{(2\pi)^d}\hat{f}(\xi)e^{i\xi \cdot x}\E_B\{e^{i\tilde{M}_t^\eps}(\langle \tilde{M}^\eps\rangle_t-(|\xi|^2+\sigma_\lambda^2)t)\}d\xi,
\end{equation}
so
\begin{equation}
\begin{aligned}
\frac{u_\eps(t,x)-u_{hom}(t,x)}{\eps}\approx &\frac{v_{1,\eps}+v_{2,\eps}}{\eps}\\
=&\int_{\R^d}\frac{1}{(2\pi)^d}\hat{f}(\xi)e^{i\xi \cdot x}\eps^{-1}\E_B\{e^{i\tilde{M}_t^\eps}\left(iR_t^\eps-\frac12(\langle \tilde{M}^\eps\rangle_t-(|\xi|^2+\sigma_\lambda^2)t)\right)\} d\xi.
\label{eq:decom1Cord5}
\end{aligned}
\end{equation}

\section{Proof of the main theorem: $d=3$}
\label{sec:proofMa}

Now we are ready to prove the main theorem. Recall that $\tilde{M}_t^\eps=\eps \xi\cdot B_{t/\eps^2}+M_t^\eps$, and $X_t^\eps=R_t^\eps+M_t^\eps$, so
\begin{equation}
\tilde{M}_t^\eps=\eps \xi\cdot B_{t/\eps^2}+X_t^\eps-R_t^\eps.
\end{equation}

By \eqref{eq:EstiR} and \eqref{eq:EstiQuadra}, $\E\E_B\{(\eps^{-\frac12}R_t^\eps)^2\}$ and $\E\E_B\{(\eps^\frac32\int_0^{t/\eps^2}D_k\Phi_\lambda(y_s)ds)^2\}$ are both bounded, so since $R_t^\eps$ is small as in \eqref{eq:EstiR}, we can replace $\tilde{M}_t^\eps$ by $\eps \xi\cdot B_{t/\eps^2}+X_t^\eps$ in \eqref{eq:decom1Cor} and obtain
\begin{equation}
\begin{aligned}
&\frac{u_\eps(t,x)-u_{hom}(t,x)}{\eps^\frac12}\\
\approx &\E_B\{\int_{\R^d}\frac{1}{(2\pi)^d}\hat{f}(\xi)e^{i\xi \cdot x+i\eps \xi\cdot B_{t/\eps^2}}e^{iX_t^\eps}\left(\eps^{-\frac12}iR_t^\eps-\eps^{\frac32}\int_0^{t/\eps^2}\sum_{k=1}^d \xi_k D_k\Phi_\lambda(y_s)ds\right)d\xi\}.
\end{aligned}
\end{equation}
Let
\begin{equation}
Y_t^\eps:=i\eps^{-\frac12}\left(\eps\int_0^{t/\eps^2}\lambda\Phi_\lambda(y_s)ds-\eps\Phi_\lambda(y_{t/\eps^2})+\eps\Phi_\lambda(y_0)\right)-\eps^{\frac32}\int_0^{t/\eps^2}\sum_{k=1}^d \xi_k D_k\Phi_\lambda(y_s)ds,
\end{equation}
so $\E\E_B\{|Y_t^\eps|^2\}$ is uniformly bounded, and we have 
\begin{equation}\frac{u_\eps(t,x)-u_{hom}(t,x)}{\eps^\frac12}\approx \E_B\{\int_{\R^d}\frac{1}{(2\pi)^d}\hat{f}(\xi)e^{i\xi \cdot x+i\eps \xi\cdot B_{t/\eps^2}}e^{iX_t^\eps}Y_t^\eps d\xi\}.
\end{equation}

We show the interaction between $X_t^\eps$ and $Y_t^\eps$ goes to zero in the following sense:
\begin{proposition}
\begin{equation}
\E_B\{\int_{\R^d}\hat{f}(\xi)e^{i\xi \cdot x+i\eps \xi\cdot B_{t/\eps^2}}(e^{iX_t^\eps}-e^{-\frac12\sigma^2t})Y_t^\eps d\xi\}\to 0
\end{equation}
in $L^2(\Omega)$ as $\eps \to 0$.
\label{prop:asymInde}
\end{proposition}

By the above Proposition, the rescaled corrector can be written as
\begin{equation}
\frac{u_\eps(t,x)-u_{hom}(t,x)}{\eps^\frac12}\approx
\E_B\{\int_{\R^d}\frac{1}{(2\pi)^d}\hat{f}(\xi)e^{i\xi \cdot x+i\eps \xi\cdot B_{t/\eps^2}}e^{-\frac12\sigma^2t}Y_t^\eps d\xi\}.
\end{equation}

For the last term in $Y_t^\eps$, we can write
\begin{equation}
\begin{aligned}
&\E_B\{\int_{\R^d}\frac{1}{(2\pi)^d}\hat{f}(\xi)e^{i\xi \cdot x+i\eps \xi\cdot B_{t/\eps^2}}e^{-\frac12\sigma^2t}\eps^{\frac32}\int_0^{t/\eps^2}\sum_{k=1}^d \xi_k D_k\Phi_\lambda(y_s)ds\}\\
=&-i\sum_{k=1}^d \E_B\{\partial_{x_k}f(x+\eps B_{t/\eps^2})e^{-\frac12\sigma^2t}\eps^{\frac32}\int_0^{t/\eps^2}D_k\Phi_\lambda(y_s)ds\}\\
=&-i\E_B\{f(x+\eps B_{t/\eps^2})e^{-\frac12\sigma^2t}\eps^{\frac12}\sum_{k=1}^d\int_0^{t/\eps^2}D_k\Phi_\lambda(y_s)dB_s^k\},
\end{aligned}
\label{eq:useMalliavin}
\end{equation}
where the last equality comes from a simple application of the duality relation in Malliavin calculus \cite{nualart2006Malliavin}. For the sake of convenience, we present some standard facts about Malliavin calculus in Appendix \ref{sec:malliavin}.

To summarize, we have
\begin{equation}
\begin{aligned}
&\frac{u_\eps(t,x)-u_{hom}(t,x)}{\eps^\frac12}\\
\approx &\E_B\{f(x+\eps B_{t/\eps^2})e^{-\frac12\sigma^2t}
i\eps^{-\frac12}\left(\eps\int_0^{t/\eps^2}\lambda\Phi_\lambda(y_s)ds-\eps\Phi_\lambda(y_{t/\eps^2})+\eps\Phi_\lambda(y_0)\right)\}\\
&+\E_B\{f(x+\eps B_{t/\eps^2})e^{-\frac12\sigma^2t}i\eps^{\frac12}\sum_{k=1}^d \int_0^{t/\eps^2}D_k\Phi_\lambda(y_s)dB_s^k\}\\
=& \E_B\{f(x+\eps B_{t/\eps^2})e^{-\frac12\sigma^2t} i\eps^{-\frac12}\eps\int_0^{t/\eps^2}\V(y_s)ds\}\\
=& \E_B\{f(x+B_t)e^{-\frac12\sigma^2t}i\frac{1}{\eps^{\frac32}}\int_0^t V(\frac{x+B_s}{\eps})ds\},
\end{aligned}
\label{eq:finalRefineC}
\end{equation}
which combines with the following Proposition to complete the proof of Theorem \ref{thm:mainTH}.
\begin{proposition}
$\E_B\{f(x+B_t)e^{-\frac12\sigma^2t}i\eps^{-\frac32}\int_0^tV(\frac{x+B_s}{\eps})ds\}\Rightarrow v(t,x)$ in the sense of Theorem \ref{thm:mainTH}, and $v(t,x)$ solves the following SPDE with additive white noise $\dot{W}(x)$ and zero initial condition:
\begin{equation}
\partial_t v(t,x)=\frac12\Delta v(t,x)-\frac12\sigma^2 v(t,x)+i\sqrt{\hat{R}(0)}u_{hom}(t,x)\dot{W}(x).
\label{eq:limitSPDE}
\end{equation}
\label{prop:wkSPDE}
\end{proposition}

\begin{remark}
To combine \eqref{eq:finalRefineC} and Proposition \ref{prop:wkSPDE} to prove Theorem \ref{thm:mainTH}, we need to note that the statistical error caused in \eqref{eq:finalRefineC} is $x-$independent, i.e.,
\begin{equation}
\E\{|\frac{u_\eps(t,x)-u_{hom}(t,x)}{\eps^\frac12}-\E_B\{f(x+B_t)e^{-\frac12\sigma^2t}i\frac{1}{\eps^{\frac32}}\int_0^t V(\frac{x+B_s}{\eps})ds\}|\}\leq C_\eps
\end{equation}
for some $x-$independent constant $C_\eps\to 0$ as $\eps\to 0$.
\end{remark}

\section{Asymptotic independence, proof of Proposition \ref{prop:asymInde}: $d=3$}
\label{sec:Inde}

Our goal is to prove that $\E_B\{\int_{\R^d}\hat{f}(\xi)e^{i\xi \cdot x+i\eps \xi\cdot B_{t/\eps^2}}(e^{iX_t^\eps}-e^{-\frac12\sigma^2t})Y_t^\eps d\xi\}\to 0$ in probability, and since $X_t^\eps, Y_t^\eps$ both depend on the Brownian motion $B_t$, we write them as $X_t^\eps(B), Y_t^\eps(B)$ and calculate the second moment
\begin{equation}
\begin{aligned}
&\E\{|\E_B\{\int_{\R^d}\hat{f}(\xi)e^{i\xi \cdot x+i\eps \xi\cdot B_{t/\eps^2}}(e^{iX_t^\eps}-e^{-\frac12\sigma^2t})Y_t^\eps d\xi\}|^2\}\\
=&\E\E_B\E_W \int_{\R^{2d}}\hat{f}(\xi)e^{i\xi\cdot x+i\eps \xi\cdot B_{t/\eps^2}}(e^{iX_t^\eps(B)}-e^{-\frac12\sigma^2t})Y_t^\eps(B)\\&\quad\quad\quad\quad\quad\quad\overline{\hat{f}(\eta)e^{i\eta\cdot x+i\eps \eta\cdot W_{t/\eps^2}}(e^{iX_t^\eps(W)}-e^{-\frac12\sigma^2t})Y_t^\eps(W)}d\xi d\eta,
\end{aligned}
\end{equation}
where $B,W$ are independent Brownian motions. We claim that 
\begin{equation}
\E_B\E_W|\E\{(e^{iX_t^\eps(B)}-e^{-\frac12\sigma^2t})Y_t^\eps(B)\overline{(e^{iX_t^\eps(W)}-e^{-\frac12\sigma^2t})Y_t^\eps(W)}\}|\to 0
\label{eq:keyEsti}
\end{equation} 
as $\eps \to 0$. If the claim is true, then
\begin{equation}
\begin{aligned}
&\E\{|\E_B\{\int_{\R^d}\hat{f}(\xi)e^{i\xi \cdot x+i\eps \xi\cdot B_{t/\eps^2}}(e^{iX_t^\eps}-e^{-\frac12\sigma^2t})Y_t^\eps d\xi\}|^2\}\\
\leq &\int_{\R^{2d}}|\hat{f}(\xi)\hat{f}(\eta)|\E_B\E_W|\E\{(e^{iX_t^\eps(B)}-e^{-\frac12\sigma^2t})Y_t^\eps(B)\overline{(e^{iX_t^\eps(W)}-e^{-\frac12\sigma^2t})Y_t^\eps(W)}\}|d\xi d\eta\to 0
\end{aligned}
\label{eq:con2ndMoment}
\end{equation}
and Proposition \ref{prop:asymInde} is proved.

\begin{remark}
From the expressions of $X_t^\eps,Y_t^\eps$, it is clear that the dependence of $$\E_B\E_W|\E\{(e^{iX_t^\eps(B)}-e^{-\frac12\sigma^2t})Y_t^\eps(B)\overline{(e^{iX_t^\eps(W)}-e^{-\frac12\sigma^2t})Y_t^\eps(W)}\}|$$ on $\xi,\eta$ is only a factor of $\xi_k,\eta_k$. Since $f\in \C_c^\infty(\R^d)$, by the dominated convergence theorem, we obtain Proposition \ref{prop:asymInde}.
\end{remark}

Therefore we only need to prove \eqref{eq:keyEsti} holds. Clearly, when freezing $B$ and $W$, $X_t^\eps,Y_t^\eps$ are Gaussian if $V$ is Gaussian, and are Poissonian if $V$ is Poissonian, which makes the explicit calculation feasible.

\subsection{Poissonian case}

Recall that $$Y_t^\eps(B)=i\eps^{-\frac12}\left(\eps\int_0^{t/\eps^2}\lambda\Phi_\lambda(y_s)ds-\eps\Phi_\lambda(y_{t/\eps^2})+\eps\Phi_\lambda(y_0)\right)-\eps^{\frac32}\int_0^{t/\eps^2}\sum_{k=1}^d \xi_k D_k\Phi_\lambda(y_s)ds,$$ and since $V(x)=\int_{\R^d}\varphi(x-y)\omega(dy)-c_\varphi$ in the Poissonian case, we have
\begin{equation}
\begin{aligned}
Y_t^\eps(B)=&i\int_{\R^d}\eps^{\frac52}\int_0^{t/\eps^2}f^\lambda(\frac{x}{\eps}+B_s-y)ds\omega(dy)-i\eps^\frac12\int_{\R^d}f^\lambda(\frac{x}{\eps}+B_{t/\eps^2}-y)\omega(dy)\\
+&i\eps^\frac12\int_{\R^d}f^\lambda(\frac{x}{\eps}-y)\omega(dy)-\sum_{k=1}^d\xi_k\int_{\R^d}\eps^{\frac32}\int_0^{t/\eps^2}f^\lambda_k(\frac{x}{\eps}+B_s-y)ds\omega(dy)-C,
\end{aligned}
\end{equation}
for some constant $C$, where $f^\lambda(x)=\int_{\R^d}\varphi(x-y)G_\lambda(y), f_k^\lambda(x)=\int_{\R^d}\varphi(x-y)\partial_{x_k}G_\lambda(y)dy$, and $C$ is chosen so that $\E\{Y_t^\eps(B)\}=0$. Therefore, $Y_t^\eps(B)=\int_{\R^d}h_B(y)\omega(dy)-\int_{\R^d}h_B(y)dy$ for some $h_B$ depending on the Brownian path $B_s,s\in [0,t/\eps^2]$.

Similarly, $X_t^\eps(B)=\eps\int_0^{t/\eps^2}V(\frac{x}{\eps}+B_s)ds=\int_{\R^d}\eps\int_0^{t/\eps^2}\varphi(\frac{x}{\eps}+B_s-y)ds\omega(dy)-c_\varphi t/\eps$,
and we denote $X_t^\eps(B)=\int_{\R^d}g_B(y)\omega(dy)-\int_{\R^d}g_B(y)dy$ for some real $g_B$ depending on the Brownian path $B_s,s\in [0,t/\eps^2]$.

To calculate $\E\{(e^{iX_t^\eps(B)}-e^{-\frac12\sigma^2t})Y_t^\eps(B)\overline{(e^{iX_t^\eps(W)}-e^{-\frac12\sigma^2t})Y_t^\eps(W)}\}$, since $X_t^\eps(B)$, $Y_t^\eps(B)$, $X_t^\eps(W)$, $Y_t^\eps(W)$ are all integrals with respect to the Poissonian point process $\omega(dy)$, we apply Lemma \ref{lem:PoissonMoment} to obtain
\begin{equation}
\begin{aligned}
&\E\{(e^{iX_t^\eps(B)}-e^{-\frac12\sigma^2t})Y_t^\eps(B)\overline{(e^{iX_t^\eps(W)}-e^{-\frac12\sigma^2t})Y_t^\eps(W)}\}\\
=&e^{-\sigma^2 t}\int_{\R^d}h_B\overline{h_W}dy\\
-&e^{-\frac12\sigma^2 t}e^{\int_{\R^d}(e^{ig_B}-1-ig_B)dy}\left(\int_{\R^d}e^{ig_B}h_B\overline{h_W}dy+\int_{\R^d}(e^{ig_B}-1)h_Bdy\int_{\R^d}(e^{ig_B}-1)\overline{h_W}dy\right)\\
-&e^{-\frac12\sigma^2 t}e^{\int_{\R^d}(e^{-ig_W}-1+ig_W)dy}\left(\int_{\R^d}e^{-ig_W}h_B\overline{h_W}dy+\int_{\R^d}(e^{-ig_W}-1)h_Bdy\int_{\R^d}(e^{-ig_W}-1)\overline{h_W}dy\right)\\
+&e^{\int_{\R^d}(e^{ig_B-ig_W}-1-ig_B+ig_W)dy}\left(\int_{\R^d}e^{ig_B-ig_W}h_B\overline{h_W}dy+\int_{\R^d}(e^{ig_B-ig_W}-1)h_Bdy\int_{\R^d}(e^{ig_B-ig_W}-1)\overline{h_W}dy\right).
\end{aligned}
\end{equation}

Let $\E\{(e^{iX_t^\eps(B)}-e^{-\frac12\sigma^2t})Y_t^\eps(B)\overline{(e^{iX_t^\eps(W)}-e^{-\frac12\sigma^2t})Y_t^\eps(W)}\}=P_1+P_2$, where
\begin{equation}
\begin{aligned}
P_1=&\left(e^{-\sigma^2t}+e^{\int_{\R^d}(e^{ig_B-ig_W}-1-ig_B+ig_W)dy}\right)\int_{\R^d}h_B\overline{h_W}dy\\
-&\left(e^{-\frac12\sigma^2 t}e^{\int_{\R^d}(e^{ig_B}-1-ig_B)dy}+e^{-\frac12\sigma^2 t}e^{\int_{\R^d}(e^{-ig_W}-1+ig_W)dy}\right)\int_{\R^d}h_B\overline{h_W}dy,
\end{aligned}
\end{equation}
and $P_2$ is the remainder, we have the following lemma concerning $P_1$.

\begin{lemma}
$\E_B\E_W\{|P_1|\}\to 0$.
\end{lemma}
\begin{proof}
Firstly, we have
\begin{equation}
\begin{aligned}
&\E_B\E_W\{|P_1|\}\\
\leq &\sqrt{\E_B\E_W\left(e^{-\sigma^2t}+e^{\int_{\R^d}(e^{ig_B-ig_W}-1-ig_B+ig_W)dy}-e^{-\frac12\sigma^2 t}e^{\int_{\R^d}(e^{ig_B}-1-ig_B)dy}-e^{-\frac12\sigma^2 t}e^{\int_{\R^d}(e^{-ig_W}-1+ig_W)dy}\right)^2}\\
&\times \sqrt{\E_B\E_W\{|\int_{\R^d}h_B\overline{h_W}dy|^2\}}.
\end{aligned}
\end{equation}
Clearly, $\E\{|Y_t^\eps(B)|^2\}=\int_{\R^d}|h_B|^2dy$ and $\E\{|Y_t^\eps(W)|^2\}=\int_{\R^d}|h_W|^2dy$, thus $$\E_B\E_W\{|\int_{\R^d}h_B\overline{h_W}dy|^2\}\leq \E\E_B\{|Y_t^\eps(B)|^2\}\E\E_W\{|Y_t^\eps(W)|^2\}$$ is uniformly bounded. Then we only have to apply Lemma \ref{lem:conPoissonChr} to complete the proof.
\end{proof}

The rest is to prove that $\E_B\E_W\{|P_2|\}\to 0$. Actually, by the fact that $e^{\int_{\R^d}(e^{ig_B}-1-ig_B)dy}$, $e^{\int_{\R^d}(e^{-ig_W}-1+ig_W)dy}$ and $e^{\int_{\R^d}(e^{ig_B-ig_W}-1-ig_B+ig_W)dy}$ are uniformly bounded by $1$, it suffices to show that in $L^1(B\times W)$
\begin{equation}
\int_{\R^d}(e^{ig_B}-1)h_Bdy\int_{\R^d}(e^{ig_B}-1)\overline{h_W}dy\to 0,
\end{equation}
\begin{equation}
\int_{\R^d}(e^{-ig_W}-1)h_Bdy\int_{\R^d}(e^{-ig_W}-1)\overline{h_W}dy\to 0,
\end{equation}
\begin{equation}
\int_{\R^d}(e^{ig_B-ig_W}-1)h_Bdy\int_{\R^d}(e^{ig_B-ig_W}-1)\overline{h_W}dy\to 0,
\label{eq:highOrderPoisson}
\end{equation}
\begin{equation}
\int_{\R^d}(e^{ig_B}-1)h_B\overline{h_W}dy\to 0,
\end{equation}
\begin{equation}
\int_{\R^d}(e^{-ig_W}-1)h_B\overline{h_W}dy\to 0,
\end{equation}
\begin{equation}
\int_{\R^d}(e^{ig_B-ig_W}-1)h_B\overline{h_W}dy\to 0.
\end{equation}

The methods to prove all the above estimates are similar, i.e., we expand $e^{ix}$ in power series and control each term after standard changes of variables. We will only present a detailed proof of \eqref{eq:highOrderPoisson} since it contains all the ingredients and the other terms are handled in a similar way.

Without loss of generality, we can assume $|\varphi(x)|$ is some bounded, radially symmetric and decreasing function with compact support in the estimation.

Since $Y_t^\eps(B)=\int_{\R^d}h_B(y)\omega(dy)-\int_{\R^d}h_B(y)dy$ with
\begin{equation}
\begin{aligned}
h_B(y)=&i\eps^{\frac52}\int_0^{t/\eps^2}f^\lambda(\frac{x}{\eps}+B_s-y)ds-\sum_{k=1}^d\xi_k\eps^\frac32\int_0^{t/\eps^2}f_k^\lambda(\frac{x}{\eps}+B_s-y)ds\\
+&i\eps^\frac12f^\lambda(\frac{x}{\eps}-y)-i\eps^\frac12f^\lambda(\frac{x}{\eps}+B_{t/\eps^2}-y)
\end{aligned}
\end{equation}
we can divide the terms into two groups depending on whether they involve the integration in $s$, i.e.,
\begin{eqnarray*}
A_1(B)&=&\{\eps^{\frac52}\int_0^{t/\eps^2}f^\lambda(\frac{x}{\eps}+B_s-y)ds, \eps^\frac32\int_0^{t/\eps^2}f_k^\lambda(\frac{x}{\eps}+B_s-y)ds\},\\
A_2(B)&=&\{\eps^\frac12f^\lambda(\frac{x}{\eps}-y),\eps^\frac12f^\lambda(\frac{x}{\eps}+B_{t/\eps^2}-y)\}.
\end{eqnarray*}
A similar decomposition holds for $h_B(W)$. To prove $\E_B\E_W\{|\int_{\R^d}(e^{ig_B-ig_W}-1)h_Bdy\int_{\R^d}(e^{ig_B-ig_W}-1)\overline{h_W}dy|\}\to 0$, there are four groups of terms concerning $h_B,h_W$ to deal with, i.e., $(A_1(B),A_1(W))$, $(A_1(B),A_2(W))$, $(A_2(B),A_1(W))$, and $(A_2(B),A_2(W))$. In the following, we will analyze them separately.

\subsubsection{$(A_1(B),A_1(W))$}

\begin{lemma}
$$\E_B\E_W\{|\int_{\R^d}(e^{ig_B-ig_W}-1)h_Bdy\int_{\R^d}(e^{ig_B-ig_W}-1)\overline{h_W}dy|\}\to 0$$ as $\eps\to 0$ for $h_B=\int_0^{t/\eps^2}g_1(\frac{x}{\eps}+B_s-y)ds$ and $\overline{h_W}=\int_0^{t/\eps^2}g_2(\frac{x}{\eps}+W_s-y)ds$, where $g_1,g_2\in \{\eps^\frac52f^\lambda,\eps^\frac32f_k^\lambda\}$.
\label{lem:A1A1}
\end{lemma}

Firstly, we write
\begin{equation}
\begin{aligned}
&\int_{\R^d}(e^{ig_B-ig_W}-1)h_Bdy\int_{\R^d}(e^{ig_B-ig_W}-1)\overline{h_W}dy\\
=&\sum_{m_1,m_2,m_3,m_4\geq 0}\frac{i^{m_1-m_2+m_3-m_4}}{m_1!m_2!m_3!m_4!}\int_{\R^{2d}}g_B(y)^{m_1}g_W(y)^{m_2}g_B(z)^{m_3}g_W(z)^{m_4}h_B(y)\overline{h_W}(z)dydz,
\end{aligned}
\end{equation}
and clearly the indexes satisfy $m_1+m_2\geq 1$ and $m_3+m_4\geq 1$. For each term, let $N(m_i)=\sum_{i=1}^4m_i$, and we have
\begin{equation}
\begin{aligned}
&|\int_{\R^{2d}}g_B(y)^{m_1}g_W(y)^{m_2}g_B(z)^{m_3}g_W(z)^{m_4}h_B(y)\overline{h_W}(z)dydz|\\
\les&\eps^{N(m_i)}\int_{\R^{2d}}\int_{[0,t/\eps^2]^{N(m_i)+2}}\prod_{i=1}^{m_1}|\varphi|(B_{s_i}-y)\prod_{i=m_1+1}^{m_1+m_3}|\varphi|(B_{s_i}-z)|g_1|(B_{\tilde{s}}-y)\\
&\prod_{i=1}^{m_2}|\varphi|(W_{u_i}-y)\prod_{i=m_2+1}^{m_2+m_4}|\varphi|(W_{u_i}-z)|g_2|(W_{\tilde{u}}-z)dsdud\tilde{s}d\tilde{u}dydz,
\end{aligned}
\end{equation}
where we have changed variables $y\to y+\frac{x}{\eps}$ and $z\to z+\frac{x}{\eps}$.

Since $m_1+m_2\geq 1$ and $m_3+m_4\geq 1$, there are four cases.
\begin{enumerate}
\item $m_1m_3\neq 0$.
\item $m_2m_4\neq 0$.
\item $m_2=m_3=0$.
\item $m_1=m_4=0$.
\end{enumerate}

\begin{lemma}
In all four cases, we have
\begin{equation}
\begin{aligned}
&\eps^{N(m_i)}\E_B\E_W\int_{\R^{2d}}\int_{[0,t/\eps^2]^{N(m_i)+2}}\prod_{i=1}^{m_1}|\varphi|(B_{s_i}-y)\prod_{i=m_1+1}^{m_1+m_3}|\varphi|(B_{s_i}-z)|g_1|(B_{\tilde{s}}-y)\\
&\quad\quad\quad\quad\quad\prod_{i=1}^{m_2}|\varphi|(W_{u_i}-y)\prod_{i=m_2+1}^{m_2+m_4}|\varphi|(W_{u_i}-z)|g_2|(W_{\tilde{u}}-z)dsdud\tilde{s}d\tilde{u}dydz\\
\leq &\eps^{N(m_i)-1}|\log\eps|^2C^{N(m_i)+2}(m_1+m_3+2)!(m_2+m_4+2)!
\end{aligned}
\end{equation}
for some constant $C>0$.
\label{lem:m1m3A1A1}
\end{lemma}

The proof of Lemma \ref{lem:m1m3A1A1} is left in the Appendix.

Now we only have to note that
\begin{equation}
\sum_{m_1+m_2\geq 1,m_3+m_4\geq 1}\frac{\eps^{N(m_i)-1}|\log\eps|^2}{m_1!m_2!m_3!m_4!}C^{N(m_i)+2}(m_1+m_3+2)!(m_2+m_4+2)!\to 0
\end{equation}
as $\eps\to 0$ to complete the proof of Lemma \ref{lem:A1A1}.



\subsubsection{$(A_2(B),A_2(W))$}

\begin{lemma}
$$\E_B\E_W\{|\int_{\R^d}(e^{ig_B-ig_W}-1)h_Bdy\int_{\R^d}(e^{ig_B-ig_W}-1)\overline{h_W}dy|\}\to 0$$ as $\eps\to 0$ for $h_B=\eps^\frac12f^\lambda(\frac{x}{\eps}+\tilde{B}-y)$ with $\tilde{B}\in \{0, B_{t/\eps^2}\}$, and $\overline{h_W}=\eps^\frac12f^\lambda(\frac{x}{\eps}+\tilde{W}-y)$ with $\tilde{W}\in \{0,W_{t/\eps^2}\}$.
\label{lem:A2A2}
\end{lemma}

Again, we expand $e^{ix}$ in power series to obtain
\begin{equation}
\begin{aligned}
&\int_{\R^d}(e^{ig_B-ig_W}-1)h_Bdy\int_{\R^d}(e^{ig_B-ig_W}-1)\overline{h_W}dy\\
=&\sum_{m_1,m_2,m_3,m_4\geq 0}\frac{i^{m_1-m_2+m_3-m_4}}{m_1!m_2!m_3!m_4!}\int_{\R^{2d}}g_B(y)^{m_1}g_W(y)^{m_2}g_B(z)^{m_3}g_W(z)^{m_4}h_B(y)\overline{h_W}(z)dydz,
\end{aligned}
\end{equation}
with $m_1+m_2\geq 1, m_3+m_4\geq 1$. For each term, let $N(m_i)=\sum_{i=1}^4m_i$, we have
\begin{equation}
\begin{aligned}
&|\int_{\R^{2d}}g_B(y)^{m_1}g_W(y)^{m_2}g_B(z)^{m_3}g_W(z)^{m_4}h_B(y)\overline{h_W}(z)dydz|\\
\les&\eps^{N(m_i)+1}\int_{\R^{2d}}\int_{[0,t/\eps^2]^{N(m_i)}}\prod_{i=1}^{m_1}|\varphi|(B_{s_i}-y)\prod_{i=m_1+1}^{m_1+m_3}|\varphi|(B_{s_i}-z)|f^\lambda|(\tilde{B}-y)\\
&\quad\quad\quad\quad\prod_{i=1}^{m_2}|\varphi|(W_{u_i}-y)\prod_{i=m_2+1}^{m_2+m_4}|\varphi|(W_{u_i}-z)|f^\lambda|(\tilde{W}-z)dsdudydz.
\end{aligned}
\end{equation}
Applying the following lemmas, the proof of Lemma \ref{lem:A2A2} is complete.

\begin{lemma}
When $N(m_i)\geq 4$, we have
\begin{equation}
\begin{aligned}
&\eps^{N(m_i)+1}\E_B\E_W\int_{\R^{2d}}\int_{[0,t/\eps^2]^{N(m_i)}}\prod_{i=1}^{m_1}|\varphi|(B_{s_i}-y)\prod_{i=m_1+1}^{m_1+m_3}|\varphi|(B_{s_i}-z)|f^\lambda|(\tilde{B}-y)\\
&\quad\quad\quad\quad\quad\quad\quad\prod_{i=1}^{m_2}|\varphi|(W_{u_i}-y)\prod_{i=m_2+1}^{m_2+m_4}|\varphi|(W_{u_i}-z)|f^\lambda|(\tilde{W}-z)dsdudydz\\
\leq & (m_1+m_3)!(m_2+m_4)!C^{N(m_i)}\eps^{N(m_i)-3}.
\end{aligned}
\end{equation}
\label{lem:sumG4A2A2}
\end{lemma}

\begin{lemma}
When $N(m_i)=2,3$, we have
\begin{equation}
\E_B\E_W|\int_{\R^{2d}}g_B(y)^{m_1}g_W(y)^{m_2}g_B(z)^{m_3}g_W(z)^{m_4}h_B(y)\overline{h_W}(z)dydz|\les \eps.
\end{equation}
\label{lem:sumL4A2A2}
\end{lemma}

The proofs of Lemma \ref{lem:sumG4A2A2} and \ref{lem:sumL4A2A2} are left in the Appendix.

\subsubsection{$(A_1(B),A_2(W))$ and $(A_2(B),A_1(W))$}

By symmetry, we only analyze $(A_1(B),A_2(W))$, i.e.,
\begin{lemma}
$$\E_B\E_W\{|\int_{\R^d}(e^{ig_B-ig_W}-1)h_Bdy\int_{\R^d}(e^{ig_B-ig_W}-1)\overline{h_W}dy|\}\to 0$$ as $\eps\to 0$ for $h_B=\int_0^{t/\eps^2}g(\frac{x}{\eps}+B_s-y)ds$ and $\overline{h_W}=\eps^\frac12f^\lambda(\frac{x}{\eps}+\tilde{W}-y)$, where $g\in \{\eps^\frac52f^\lambda,\eps^\frac32f_k^\lambda\}$, and $\tilde{W}\in \{0,W_{t/\eps^2}\}$.
\label{lem:A1A2}
\end{lemma}

Similarly, we have
\begin{equation}
\begin{aligned}
&\int_{\R^d}(e^{ig_B-ig_W}-1)h_Bdy\int_{\R^d}(e^{ig_B-ig_W}-1)\overline{h_W}dy\\
=&\sum_{m_1+m_2\geq 1,m_3+m_4\geq 1}\frac{i^{m_1-m_2+m_3-m_4}}{m_1!m_2!m_3!m_4!}\int_{\R^{2d}}g_B(y)^{m_1}g_W(y)^{m_2}g_B(z)^{m_3}g_W(z)^{m_4}h_B(y)\overline{h_W}(z)dydz,
\end{aligned}
\end{equation}
and let $N(m_i)=\sum_{i=1}^4m_i$, the following two lemmas suffice to show Lemma \ref{lem:A1A2}.
\begin{lemma}
If $N(m_i)\geq 3$, then
\begin{equation}
\begin{aligned}
&\eps^{N(m_i)+\frac12}\E_B\E_W\int_{\R^{2d}}\int_{[0,t/\eps^2]^{N(m_i)+1}}\prod_{i=1}^{m_1}|\varphi|(B_{s_i}-y)\prod_{i=m_1+1}^{m_1+m_3}|\varphi|(B_{s_i}-z)|g|(B_s-y)\\
&\quad\quad\quad\quad\quad\quad\quad\prod_{i=1}^{m_2}|\varphi|(W_{u_i}-y)\prod_{i=m_2+1}^{m_2+m_4}|\varphi|(W_{u_i}-z)|f^\lambda|(\tilde{W}-z)dsdudydz\\
\leq & \eps^{N(m_i)-2}|\log\eps|(m_2+m_4+1)!(m_1+m_3+1)!.
\end{aligned}
\end{equation}
\label{lem:sumG3A1A2}
\end{lemma}

\begin{lemma}
If $N(m_i)=2$, then
\begin{equation}
\E_B\E_W |\int_{\R^{2d}}g_B(y)^{m_1}g_W(y)^{m_2}g_B(z)^{m_3}g_W(z)^{m_4}h_B(y)\overline{h_W}(z)dydz|\les \eps|\log \eps|.
\end{equation}
\label{lem:sumL2A1A2}
\end{lemma}
The proofs of Lemma \ref{lem:sumG3A1A2} and \ref{lem:sumL2A1A2} are left in the Appendix.

\subsection{Gaussian case}

When $V$ is Gaussian, $X_t(B),Y_t(B),X_t(W),Y_t(W)$ are all Gaussian when freezing $B,W$, and our goal is to prove \eqref{eq:keyEsti} by an explicit calculation and estimation of $$\E\{(e^{iX_t^\eps(B)}-e^{-\frac12\sigma^2t})Y_t^\eps(B)\overline{(e^{iX_t^\eps(W)}-e^{-\frac12\sigma^2t})Y_t^\eps(W)}\}.$$ 

If $(N_1,N_2,N_3,N_4)$ are jointly Gaussian with zero mean and covariance matrix $\Sigma$, by explicit calculation, we have
\begin{equation}
\begin{aligned}
&\E\{(e^{iN_1}-e^{-\frac12\sigma^2t})(e^{iN_2}-e^{-\frac12\sigma^2t})N_3N_4\}\\
=& \Sigma_{31}\Sigma_{41}\left(e^{-\frac12\sigma^2t}e^{-\frac12\Sigma_{11}}-e^{-\frac12\Sigma_{11}-\frac12\Sigma_{22}-\Sigma_{12}}\right)\\
+&\Sigma_{32}\Sigma_{42}\left(e^{-\frac12\sigma^2t}e^{-\frac12\Sigma_{22}}-e^{-\frac12\Sigma_{11}-\frac12\Sigma_{22}-\Sigma_{12}}\right)\\
-&\Sigma_{32}\Sigma_{41}e^{-\frac12\Sigma_{11}-\frac12\Sigma_{22}-\Sigma_{12}}\\
-&\Sigma_{31}\Sigma_{42}e^{-\frac12\Sigma_{11}-\frac12\Sigma_{22}-\Sigma_{12}}\\
+&\Sigma_{34}\left(e^{-\sigma^2t}+e^{-\frac12\Sigma_{11}-\frac12\Sigma_{22}-\Sigma_{12}}-e^{-\frac12\sigma^2t}e^{-\frac12\Sigma_{11}}-e^{-\frac12\sigma^2t}e^{-\frac12\Sigma_{22}}\right).
\end{aligned}
\end{equation}

Let $\Sigma$ be the covariance matrix of $(X_t^\eps(B),-X_t^\eps(W),Y_t^\eps(B),\overline{Y_t^\eps(W)})$, and $$\E\{(e^{iX_t^\eps(B)}-e^{-\frac12\sigma^2t})Y_t^\eps(B)\overline{(e^{iX_t^\eps(W)}-e^{-\frac12\sigma^2t})Y_t^\eps(W)}\}=P_1+P_2,$$ where
\begin{equation}
P_1=\Sigma_{34}\left(e^{-\sigma^2t}+e^{-\frac12\Sigma_{11}-\frac12\Sigma_{22}-\Sigma_{12}}-e^{-\frac12\sigma^2t}e^{-\frac12\Sigma_{11}}-e^{-\frac12\sigma^2t}e^{-\frac12\Sigma_{22}}\right),
\end{equation}
and $P_2$ is the remainder.

\begin{lemma}
$\E_B\E_W\{|P_1|\}\to 0$.
\end{lemma}
\begin{proof}
First, we have
\begin{equation}
\begin{aligned}
&\E_B\E_W\{|P_1|\}\\
\leq &\sqrt{\E_B\E_W\left(e^{-\sigma^2t}+e^{-\frac12\Sigma_{11}-\frac12\Sigma_{22}-\Sigma_{12}}-e^{-\frac12\sigma^2t}e^{-\frac12\Sigma_{11}}-e^{-\frac12\sigma^2t}e^{-\frac12\Sigma_{22}}\right)^2}\\
&\times \sqrt{\E_B\E_W|\Sigma_{34}|^2}.
\end{aligned}
\end{equation}
Clearly, $\Sigma_{34}=\E\{Y_t^\eps(B)\overline{Y_t^\eps(W)}\}$, so $\E_B\E_W|\Sigma_{34}|^2\leq \E\E_B\{|Y_t^\eps(B)|^2\}\E\E_W\{|Y_t^\eps(W)|^2\}$ is uniformly bounded. Then we only need to apply Lemma \ref{lem:conGaussChr} to complete the proof.
\end{proof}

The following lemma suffices to prove $\E_B\E_W\{|P_2|\}\to 0$.
\begin{lemma}
$\E_B\E_W\{|\Sigma_{ij}|^2\}\to 0$ with $i\in \{3,4\}, j\in \{1,2\}$.
\end{lemma}

\begin{proof}
By symmetry, we only need to prove that $\E_B\E_W\{I_n^2\}\to 0$ for
\begin{eqnarray}
I_1&=&\eps^\frac72\int_{[0,t/\eps^2]^2}(R\star G_\lambda)(x_s-B_u)dsdu,\\
I_2&=&\eps^\frac32\int_{[0,t/\eps^2]}(R\star G_\lambda)(x_{t/\eps^2}-B_u)du,\\
I_3&=&\eps^\frac32\int_{[0,t/\eps^2]}(R\star G_\lambda)(B_u)du,\\
I_4&=&\eps^{\frac52}\int_{[0,t/\eps^2]^2}(R\star \partial_{x_k}G_\lambda)(x_s-B_u)dsdu,
\end{eqnarray}
where $x\in \{B,W\}$. All cases are contained in Lemma \ref{lem:Brown2ndMo}, \ref{lem:Brown4thMo} if we replace $R$ by some bounded, integrable, positive, and radially symmetric and decreasing function, $G_\lambda$ by $e^{-c\sqrt{\lambda}|x|}|x|^{2-d}$, and $\partial_{x_k}G_\lambda$ by $e^{-c\sqrt{\lambda}|x|}|x|^{1-d}$ for some constant $c>0$. In the end, we only need to apply Lemma \ref{lem:boundConvolution} to conclude the proof.
\end{proof}

\section{Gaussian Limit, proof of Proposition \ref{prop:wkSPDE}: $d=3$}
\label{sec:GaussLi}

Let $v_\eps(t,x):=\E_B\{f(x+B_t)e^{-\frac12\sigma^2t}i\eps^{-\frac32}\int_0^tV(\frac{x+B_s}{\eps})ds\}$, we show that it is a solution to a parabolic equation with an additive potential.
\begin{lemma}
 $v_\eps(t,x)$ solves the following equation
\begin{equation}
\partial_t v_\eps(t,x)=\frac12\Delta v_\eps(t,x)-\frac12\sigma^2 v_\eps(t,x)+iu_{hom}(t,x)\frac{1}{\eps^\frac32}V(\frac{x}{\eps})
\label{eq:EqVeps}
\end{equation}
with zero initial condition.
\end{lemma}

\begin{proof}
By Feynman-Kac formula, we can write the solution to \eqref{eq:EqVeps} as
\begin{equation}
v_\eps(t,x)=\E_B\{\int_0^t e^{-\frac12\sigma^2 s}iu_{hom}(t-s,x+B_s)\frac{1}{\eps^{\frac32}}V(\frac{x+B_s}{\eps})ds\}.
\end{equation}
Since $u_{hom}$ solves the homogenized equation \eqref{eq:homoEq}, $u_{hom}(t,x)=\E_W\{f(x+W_t)e^{-\frac12\sigma^2t}\}$, so we have
\begin{equation}
\begin{aligned}
v_\eps(t,x)=&\E_B\E_W\{\int_0^t e^{-\frac12\sigma^2 s}if(x+B_s+W_{t-s})e^{-\frac12\sigma^2(t-s)}\frac{1}{\eps^{\frac32}}V(\frac{x+B_s}{\eps})ds\}\\
=&\E_B\E_W\{\int_0^t f(x+B_s+W_{t-s})e^{-\frac12\sigma^2t}i\frac{1}{\eps^\frac32}V(\frac{x+B_s}{\eps})ds\}\\
=&\E_B\{f(x+B_t)e^{-\frac12\sigma^2t}i\frac{1}{\eps^\frac32}\int_0^tV(\frac{x+B_s}{\eps})ds\}.
\end{aligned}
\end{equation}
\end{proof}

Since $v_\eps$ solves \eqref{eq:EqVeps} with zero initial condition, the solution may be written as
\begin{equation}
v_\eps(t,x)=i\int_0^t\int_{\R^d} \G_{t-s}(x-y)u_{hom}(s,y)\frac{1}{\eps^{\frac32}}V(\frac{y}{\eps})dyds,
\end{equation}
where $\G_{t-s}(x-y)=e^{-\frac12\sigma^2(t-s)}q_{t-s}(x-y)$.

We first show for fixed $(t,x)$, $v_\eps(t,x)\Rightarrow v(t,x)$ in distribution.

The solution to the limiting SPDE \eqref{eq:limitSPDE} can be written as
\begin{equation}
v(t,x)=i\int_0^t\int_{\R^d}\G_{t-s}(x-y)\sqrt{\hat{R}(0)}u_{hom}(s,y)W(dy)ds,
\end{equation}
with $W(dy)$ the Wiener integral.

Let\begin{eqnarray*}
var_\eps:&=&\int_0^t\int_0^t\int_{\R^{2d}}\G_{t-s}(x-y)\G_{t-u}(x-z)u_{hom}(s,y)u_{hom}(u,z)\frac{1}{\eps^3}R(\frac{y-z}{\eps})dydzdsdu,\\
var:&=&\hat{R}(0)\int_0^t\int_0^t\int_{\R^d}\G_{t-s}(x-z)\G_{t-u}(x-z)u_{hom}(s,z)u_{hom}(u,z)dzdsdu.
\end{eqnarray*}

\begin{lemma}
$var_\eps\to var$.
\label{lem:conVar}
\end{lemma}

\begin{proof}
By change of variables, we have
\begin{equation}
var_\eps=\int_0^t\int_0^t \int_{\R^{2d}}\G_{t-s}(x-z-\eps w)\G_{t-u}(x-z)u_{hom}(s,z+\eps w)u_{hom}(u,z)R(w)dwdzdsdu.
\end{equation}
For fixed $s,u\in (0,t)$, \begin{equation}
\begin{aligned}
&\int_{\R^{2d}}\G_{t-s}(x-z-\eps w)\G_{t-u}(x-z)u_{hom}(s,z+\eps w)u_{hom}(u,z)R(w)dwdz\\
\to &\hat{R}(0)\int_{\R^d}\G_{t-s}(x-z)\G_{t-u}(x-z)u_{hom}(s,z)u_{hom}(u,z)dz
\end{aligned}
\end{equation}
by the dominated convergence theorem. Since $u_{hom}$ is bounded, we have
\begin{equation}
|\int_{\R^{2d}}\G_{t-s}(x-z-\eps w)\G_{t-u}(x-z)u_{hom}(s,z+\eps w)u_{hom}(u,z)R(w)dwdz|\les \frac{1}{(2t-s-u)^{\frac{d}{2}}},
\end{equation}
which is integrable in $[0,t]^2$ since $d=3$. Thus again by the dominated convergence theorem, the proof is complete.
\end{proof}

If $V$ is Gaussian, then $v_\eps(t,x)$ is Gaussian. Since both the mean and variance converge, we have $v_\eps(t,x)\Rightarrow v(t,x)$ in distribution. For the convergence of finite dimensional distributions, we only need to show the convergence of $\E\{v_\eps(t_1,x_1)v_\eps(t_2,x_2)\}$, but the proof is the same as in Lemma \ref{lem:conVar}.

If $V$ is Poissonian $V(x)=\int_{\R^d}\varphi(x-y)\omega(dy)-c_\varphi$, then
\begin{equation}
\begin{aligned}
v_\eps(t,x)=&i\int_{\R^d}\left(\int_0^t\int_{\R^d}\G_{t-s}(x-y)u_{hom}(s,y)\frac{1}{\eps^\frac32}\varphi(\frac{y}{\eps}-z)dyds\right)\omega(dz)\\
-&i\int_0^t\int_{\R^d}\G_{t-s}(x-y)u_{hom}(s,y)\frac{1}{\eps^\frac32}c_\varphi dyds
\end{aligned}
\end{equation}
is Poissonian as well, and we have the following lemma.

\begin{lemma}
For any $\theta\in \R$, 
\begin{equation}
\E\{\exp(\theta v_\eps(t,x))\}\to \E\{\exp(\theta v(t,x))\}
\end{equation}
as $\eps \to 0$.
\label{lem:conChr}
\end{lemma}

\begin{proof}
Let $f_\eps(z)=\int_0^t\int_{\R^d}\G_{t-s}(x-y)u_{hom}(s,y)\eps^{-\frac32}\varphi(\frac{y}{\eps}-z)dyds$, then
\begin{equation}
\begin{aligned}
&\E\{\exp(i\theta \int_{\R^d}f_\eps(z)\omega(dz)-i\theta\int_0^t\int_{\R^d}\G_{t-s}(x-y)u_{hom}(s,y)\eps^{-\frac32}c_\varphi dyds)\}\\
=&\exp(\int_{\R^d} (e^{i\theta f_\eps(z)}-1)dz-i\theta\int_0^t\int_{\R^d}\G_{t-s}(x-y)u_{hom}(s,y)\eps^{-\frac32}c_\varphi dyds)\\
=&\exp(\int_{\R^d}\sum_{k=2}^\infty \frac{1}{k!}(i\theta)^kf_\eps(z)^kdz),
\end{aligned}
\end{equation}
since $\int_{\R^d}\varphi(z)dz=c_\varphi$.

When $k=2$, $\int_{\R^d}f_\eps(z)^2dz=var_\eps$, so by Lemma \ref{lem:conVar}, $\int_{\R^d}f_\eps(z)^2dz\to var$.

When $k\geq 3$, note that $\G_{t-s}(x-y)\leq q_{t-s}(x-y)$ and $u_{hom}$ is bounded, so we have $|f_\eps(z)|\les \int_0^t\int_{\R^d}q_s(x-y)\frac{1}{\eps^{\frac32}}|\varphi|(\frac{y}{\eps}-z)dyds$, which implies
\begin{equation}
\int_{\R^d}|f_\eps(z)|^kdz\les \frac{1}{\eps^{\frac{3k}{2}}}\int_{\R^d}\int_{[0,t]^k}\int_{\R^{kd}}\prod_{i=1}^k q_{s_i}(x-y_i)|\varphi|(\frac{y_i}{\eps}-z)dydsdz.
\end{equation}
In the Fourier domain, by change of variables and integration in $z$, we have
\begin{equation}
\begin{aligned}
\int_{\R^d}|f_\eps(z)|^kdz\les& \frac{1}{\eps^{\frac{3k}{2}}}
\int_{[0,t]^k}\int_{\R^{kd}}\prod_{i=1}^k|\F\{|\varphi|\}(\xi_i)|e^{-\frac{|\xi_i|^2}{2\eps^2}s_i}\delta(\sum_{i=1}^k \xi_i)d\xi ds\\
=&\frac{1}{\eps^{\frac{3k}{2}}}
\int_{[0,t]^k}\int_{\R^{(k-1)d}}|\F\{|\varphi|\}(-\sum_{i=2}^k\xi_i)|e^{-\frac{|\sum_{i=2}^k\xi_i|^2}{2\eps^2}s_1}\prod_{i=2}^k|\F\{|\varphi|\}(\xi_i)|e^{-\frac{|\xi_i|^2}{2\eps^2}s_i}d\xi ds.
\end{aligned}
\end{equation}
Changing variables $\xi_2\to \eps \xi_2$, $s_i\to \eps^2 s_i,i\geq 3$, and since $|\F\{|\varphi|\}|$ is uniformly bounded, we have
\begin{equation}
\begin{aligned}
\int_{\R^d}|f_\eps(z)|^kdz\les \eps^{\frac{k}{2}-1}\int_{[0,t]^2}\int_{[0,t/\eps^2]^{k-2}}\int_{\R^{(k-1)d}}&e^{-\frac12|\xi_2+\frac{1}{\eps}\sum_{i=3}^k\xi_i|^2s_1}e^{-\frac12|\xi_2|^2s_2}\\
&\prod_{i=3}^k |\F\{|\varphi|\}(\xi_i)|e^{-\frac{|\xi_i|^2}{2}s_i}d\xi ds.
\end{aligned}
\end{equation}
Clearly $\int_{\R^d}e^{-\frac12|\xi_2+\frac{1}{\eps}\sum_{i=3}^k\xi_i|^2s_1}e^{-\frac12|\xi_2|^2s_2}d\xi_2\les |s_1+s_2|^{-\frac{d}{2}}$, which is integrable in $[0,t]^2$ when $d=3$. Now we only have to integrate in $s_i, i\geq 3$ and use the fact that $\F\{|\varphi|\}(\xi)|\xi|^{-2}$ is integrable to conclude that $\int_{\R^d}|f_\eps(z)|^kdz\leq C^k \eps^{\frac{k}{2}-1}$, so
\begin{equation}
\sum_{k\geq 3}\frac{1}{k!}|\theta|^k\int_{\R^d}|f_\eps(z)|^kdz
\to 0
\end{equation}
as $\eps \to 0$. The proof is complete.
\end{proof}

To prove the convergence of finite dimensional distributions in the Poissonian case, we only need to apply the results for Gaussian when $k=2$, and use the fact that $|\sum_{i=1}^N a_i|^k\leq N^{k-1}\sum_{i=1}^N |a_i|^k$ when $k\geq 3$ in the proof of Lemma \ref{lem:conChr}.

For the convergence of \begin{equation}
\int_{\R^d}v_\eps(t,x)g(x)dx\Rightarrow \int_{\R^d}v(t,x)g(x)dx
\end{equation}
weakly for $g\in \C_c^\infty(\R^d)$, the discussion is the same as in Lemma \ref{lem:conVar} and \ref{lem:conChr}.

\section{Proof of the main theorem: $d\geq 4$}
\label{sec:THd4}

We first consider the case $d\geq 5$ when the stationary corrector exists. For constant initial condition, we will see that the discussion of the critical case $d=4$ is similar to $d\geq 5$.

The following lemma confirms that the existence of a stationary corrector is equivalent with the integrability of $\hat{R}(\xi)|\xi|^{-4}$.

\begin{lemma}
The equation $-L\Phi=\V$ has a solution in $L^2(\Omega)$ only when $\hat{R}(\xi)|\xi|^{-4}$ is integrable, and we have the regularized corrector $\Phi_\lambda\to\Phi$ in $L^2(\Omega)$.
\label{lem:exSCor}
\end{lemma}

\begin{proof}
By spectral representation, the solution should be written as
\begin{equation}
\Phi=\int_{\R^d}\frac{2}{|\xi|^2}U(d\xi)\V,
\end{equation}
and for it to be well-defined, we need 
\begin{equation}
\langle \int_{\R^d}\frac{2}{|\xi|^2}U(d\xi)\V,\int_{\R^d}\frac{2}{|\xi|^2}U(d\xi)\V\rangle=\frac{4}{(2\pi)^d}\int_{\R^d}\frac{\hat{R}(\xi)}{|\xi|^4}d\xi<\infty.
\end{equation}
If the integrability condition holds, we have
\begin{equation}
\langle \Phi_\lambda-\Phi,\Phi_\lambda-\Phi\rangle\les \int_{\R^d}\frac{\lambda^2\hat{R}(\xi)}{(2\lambda+|\xi|^2)^2|\xi|^4}d\xi\to 0
\end{equation}
by the dominated convergence theorem.
\end{proof}

Under Assumption \ref{ass:mixing}, $R(x)$ decays sufficiently fast, so $\hat{R}(\xi)$ is bounded, and the stationary corrector exists when $d\geq 5$. We recall \eqref{eq:decom1Cord5} that
\begin{equation}
\begin{aligned}
\frac{u_\eps(t,x)-u_{hom}(t,x)}{\eps}\approx &\frac{v_{1,\eps}+v_{2,\eps}}{\eps}\\
=&\int_{\R^d}\frac{1}{(2\pi)^d}\hat{f}(\xi)e^{i\xi \cdot x}\eps^{-1}\E_B\{e^{i\tilde{M}_t^\eps}\left(iR_t^\eps-\frac12(\langle \tilde{M}^\eps\rangle_t-(|\xi|^2+\sigma_\lambda^2)t)\right)\} d\xi.
\end{aligned}
\end{equation}
$v_{1,\eps}$ corresponds to the contribution from the remainder $R_t^\eps$, and $v_{2,\eps}$ corresponds to the contribution from the martingales, i.e., by the quantitative martingale central limit theorem, it reduces to the difference between quadratic variations $\langle \tilde{M}^\eps\rangle_t-(|\xi|^2+\sigma_\lambda^2)t$. 

We will analyze $v_{1,\eps}/\eps$ and $v_{2,\eps}/\eps$ separately, and it turns out that the remainder contributes to the random corrector while the martingale part contributes to the deterministic error.

\subsection{Analysis of $v_{1,\eps}$: $d\geq 5$}

Recall that 
\begin{equation}
\frac{v_{1,\eps}}{\eps}=\int_{\R^d}\frac{1}{(2\pi)^d}\hat{f}(\xi)e^{i\xi\cdot x}\eps^{-1}\E_B\{e^{i\tilde{M}_t^\eps}iR_t^\eps\}d\xi,
\end{equation}
where $R^\eps_t=\eps\int_0^{t/\eps^2}\lambda\Phi_\lambda(y_s)ds-\eps\Phi_\lambda(y_{t/\eps^2})+\eps\Phi_\lambda(y_0)$ with the environmental process $y_s=\tau_{\frac{x}{\eps}+B_s}\omega$, we discuss the three terms respectively.

\begin{lemma}
$\int_{\R^d}(2\pi)^{-d}\hat{f}(\xi)e^{i\xi\cdot x}\E_B\{e^{i\tilde{M}_t^\eps}i\Phi_\lambda(y_0)\}d\xi-iu_{hom}(t,x)\Phi(\tau_{\frac{x}{\eps}}\omega)\to 0 $ in $L^1(\Omega)$.
\label{lem:remainderT1d5}
\end{lemma}
\begin{proof}
First of all, we show that 
\begin{equation}
\int_{\R^d}\frac{1}{(2\pi)^d}\hat{f}(\xi)e^{i\xi\cdot x}\left(\E_B\{e^{i\tilde{M}_t^\eps}\}-e^{-\frac12(|\xi|^2+\sigma_\lambda^2)t}\right)i\Phi_\lambda(y_0)d\xi\to 0 
\end{equation}
in $L^1(\Omega)$. The result comes from an application of the quantitative martingale central limit theorem, together with the fact that $\E\E_B\{|\langle \tilde{M}^\eps\rangle_t-(|\xi|^2+\sigma_\lambda^2)t|^2\}\les \eps^2$ and $\E\{\Phi_\lambda^2\}$ is uniformly bounded when $d\geq 5$. Since $|\sigma_\lambda^2-\sigma^2|\les \eps^2$ and $\Phi_\lambda\to \Phi$ by Lemma \ref{lem:exSCor}, we obtain that
\begin{equation}
\int_{\R^d}\frac{1}{(2\pi)^{d}}\hat{f}(\xi)e^{i\xi\cdot x}\E_B\{e^{i\tilde{M}_t^\eps}i\Phi_\lambda(y_0)\}d\xi-
\int_{\R^d}\frac{1}{(2\pi)^{d}}\hat{f}(\xi)e^{i\xi\cdot x}e^{-\frac12(|\xi|^2+\sigma^2)t}i\Phi(y_0)d\xi\to 0
\end{equation}
in $L^1(\Omega)$. The proof is complete.
\end{proof}

\begin{lemma}
$\int_{\R^d}(2\pi)^{-d}\hat{f}(\xi)e^{i\xi\cdot x}\E_B\{e^{i\tilde{M}_t^\eps}\Phi_\lambda(y_{t/\eps^2})\}d\xi\to0$ in $L^1(\Omega)$.
\label{lem:remainderT2d5}
\end{lemma}
\begin{proof}
We only need to show that $\E_B\{e^{i\tilde{M}_t^\eps}\Phi_\lambda(y_{t/\eps^2})\}\to 0$ in $L^1(\Omega)$. Recall that $\tilde{M}_t^\eps=\sum_{k=1}^d\eps\int_0^{t/\eps^2}(\xi_k+D_k\Phi_\lambda(y_s))dB_s^k$. For any $u\in (0,t/\eps^2)$ that may depend on $\eps$, we consider 
\begin{equation}
\E_B\{e^{i\sum_{k=1}^d\eps\int_0^u(\xi_k+D_k\Phi_\lambda(y_s))dB_s^k}\Phi_\lambda(y_{t/\eps^2})\}=\E_B\{\E_B\{\Phi_\lambda(y_{t/\eps^2})|\F_u\}e^{i\sum_{k=1}^d\eps\int_0^u(\xi_k+D_k\Phi_\lambda(y_s))dB_s^k}\}
\end{equation}
with $\F_s$ the natural filtration associated with $B$. The r.h.s. of the last display can be bounded by $\E_B\{|\E_B\{\Phi_\lambda(y_{t/\eps^2})|\F_u\}|\}$, and since $y_s$ is invariant with respect to $\Pb$, we have 
\begin{equation}
\E\E_B\{|\E_B\{\Phi_\lambda(y_{t/\eps^2})|\F_u\}|\}=\E\{|\E_B\{\Phi_\lambda(y_{t/\eps^2-u})\}|\}.
\end{equation}
By an explicit calculation, we have 
\begin{equation}
\E\{|\E_B\{\Phi_\lambda(y_{s})\}|^2\}=\frac{1}{(2\pi)^d}\int_{\R^d}\widehat{R_{\Phi_\lambda}}(\xi)e^{-|\xi|^2s}d\xi\to 0
\label{eq:proofremainderT2d5}
\end{equation}
as $s\to\infty$, where $R_{\Phi_\lambda}$ is the covariance function of $\Phi_\lambda$ and satisfies $\widehat{R_{\Phi_\lambda}}(\xi)\les \hat{R}(\xi)|\xi|^{-4}$. Now we have
\begin{equation}
\begin{aligned}
\E_B\{e^{i\tilde{M}_t^\eps}\Phi_\lambda(y_{t/\eps^2})\}=&\E_B\{(e^{i\tilde{M}_t^\eps}-e^{i\sum_{k=1}^d\eps\int_0^u(\xi_k+D_k\Phi_\lambda(y_s))dB_s^k})\Phi_\lambda(y_{t/\eps^2})\}\\
+&\E_B\{e^{i\sum_{k=1}^d\eps\int_0^u(\xi_k+D_k\Phi_\lambda(y_s))dB_s^k}\Phi_\lambda(y_{t/\eps^2})\}
\end{aligned}
\end{equation}
for any $u\in (0,t/\eps^2)$. The second term goes to zero in $L^1(\Omega)$ if we choose $u$ so that $t/\eps^2-u\to \infty$ as $\eps \to 0$ by the above discussion. For the first term, its $L^1$ norm is bounded by $\sqrt{\eps^2(t/\eps^2-u)}$ since $\Phi_\lambda$ is bounded in $L^2$. Therefore, we only need to choose $u$ so that $t/\eps^2-u\to \infty$ and $\eps^2(t/\eps^2-u)\to 0$, e.g., when $t/\eps^2-u=1/\eps$ to complete the proof.
\end{proof}

\begin{lemma}
$\int_{\R^d}(2\pi)^{-d}\hat{f}(\xi)e^{i\xi\cdot x}\E_B\{e^{i\tilde{M}_t^\eps}\int_0^{t/\eps^2}\lambda\Phi_\lambda(y_s)ds\}d\xi\to0$ in $L^2(\Omega)$.
\label{lem:remainderT3d5}
\end{lemma}

\begin{proof}
We only need to show $\E\E_B\{|\int_0^{t/\eps^2}\lambda\Phi_\lambda(y_s)ds|^2\}\to 0$. By an explicitly calculation, we have that
\begin{equation}
\begin{aligned}
\E\E_B\{|\int_0^{t/\eps^2}\lambda\Phi_\lambda(y_s)ds|^2\}
=&\int_0^{t/\eps^2}\int_0^{t/\eps^2}\lambda^2\E_B\{R_{\Phi_\lambda}(B_s-B_u)\}dsdu\\
=&\frac{1}{(2\pi)^d}\int_0^{t}\int_0^{t}\int_{\R^d}\widehat{R_{\Phi_\lambda}}(\xi)e^{-\frac12|\xi|^2\frac{|s-u|}{\eps^2}}dsdu,
\end{aligned}
\end{equation}
where $R_{\Phi_\lambda}(x)$ is the covariance function of $\Phi_\lambda$. Clearly $\widehat{R_{\Phi_\lambda}}(\xi)\les \hat{R}(\xi)|\xi|^{-4}$, so by the dominated convergence theorem, the proof is complete.
\end{proof}

Combining Lemma \ref{lem:remainderT1d5}, \ref{lem:remainderT2d5} and \ref{lem:remainderT3d5}, we conclude that
\begin{equation}
\frac{v_{1,\eps}}{\eps}-iu_{hom}(t,x)\Phi(\tau_{\frac{x}{\eps}}\omega)\to 0
\label{eq:EsV1d5}
\end{equation}
in $L^1(\Omega)$.

\subsection{Analysis of $v_{2,\eps}$: $d\geq 5$}
\label{sec:conMou}

Let $$Z_{\lambda,\xi}:=2\sum_{k=1}^d\xi_kD_k\Phi_\lambda+\sum_{k=1}^d (D_k\Phi_\lambda)^2-\sigma_\lambda^2,$$ we have 
\begin{equation}
\frac{v_{2,\eps}}{\eps}=-\frac12\int_{\R^d}\frac{1}{(2\pi)^d}\hat{f}(\xi)e^{i\xi\cdot x}\E_B\{e^{i\tilde{M}_t^\eps}\eps\int_0^{t/\eps^2}Z_{\lambda,\xi}(y_s)ds\}d\xi.
\end{equation}
We will show that $Z_{\lambda,\xi}$ can be replaced by $$Z_\xi:=2\sum_{k=1}^d\xi_kD_k\Phi+\sum_{k=1}^d (D_k\Phi)^2-\sigma^2,$$ so the term $\eps\int_0^{t/\eps^2}Z_\xi(y_s)ds$ is again of the form of Brownian motion in random scenery, to which we will apply Kipnis-Varadhan's method again.

\begin{lemma}
$\E\E_B\{|\eps\int_0^{t/\eps^2}(Z_{\lambda,\xi}(y_s)-Z_\xi(y_s))ds|\}\to 0$ as $\eps\to0$.
\label{lem:martingaleT1d5}
\end{lemma}

\begin{proof}
We first have $\E\{|Z_{\lambda,\xi}-Z_\xi|\}\les \sum_{k=1}^d \|D_k\Phi_\lambda-D_k\Phi\|+|\sigma_\lambda^2-\sigma^2|$. By a straightforward calculation, we have $r.h.s.\les \eps^{\frac{d}{2}-1}1_{d\leq 5}+\eps^2|\log\eps|^\frac121_{d=6}+\eps^21_{d\geq 7}$, so when $d\geq 5$, 
\begin{equation}
\E\E_B\{|\eps\int_0^{t/\eps^2}(Z_{\lambda,\xi}(y_s)-Z_\xi(y_s))ds|\}\les \frac{1}{\eps}\E\{|Z_{\lambda,\xi}-Z_\xi|\}\to 0
\end{equation}
as $\eps\to 0$.
\end{proof}

The above proof shows that $Z_{\lambda,\xi}\to Z_\xi$ in $L^1(\Omega)$. We claim that the convergence is actually in $L^2(\Omega)$ by proving $\{Z_{\lambda,\xi}\}$ is a Cauchy sequence in $L^2(\Omega)$.

\begin{proposition}
$Z_{\lambda,\xi}\to Z_\xi$ in $L^2(\Omega)$ and $\E\E_B\{(\eps\int_0^{t/\eps^2}Z_\xi(y_s)ds)^2\}$ is bounded uniformly in $\eps$.
\label{prop:martingaleT2d5}
\end{proposition}

\begin{proof}
Since $D_k\Phi_\lambda\to D_k\Phi$ in $L^2(\Omega)$ by \cite[Proposition 3.2]{gu2013weak}, we will show the convergence in $L^2(\Omega)$ of $\sum_{k=1}^d (D_k\Phi_\lambda)^2-\sigma_\lambda^2$. It already converges in $L^1(\Omega)$ by the proof of Lemma \ref{lem:martingaleT1d5}, so we only need to show it is a Cauchy sequence in $L^2(\Omega)$ by proving $\langle \sum_{k=1}^d (D_k\Phi_{\lambda_1})^2-\sigma_{\lambda_1}^2,\sum_{k=1}^d (D_k\Phi_{\lambda_2})^2-\sigma_{\lambda_2}^2\rangle$ converges as $\lambda_1,\lambda_2\to 0$. By a direct calculation, we obtain that 
\begin{equation}
\langle \sum_{k=1}^d (D_k\Phi_{\lambda_1})^2-\sigma_{\lambda_1}^2,\sum_{k=1}^d (D_k\Phi_{\lambda_2})^2-\sigma_{\lambda_2}^2\rangle=\sum_{m,n=1}^d I_{mn}(\lambda_1,\lambda_2)-\sigma_{\lambda_1}^2\sigma_{\lambda_2}^2,
\end{equation}
where \begin{equation}
\begin{aligned}
&I_{mn}(\lambda_1,\lambda_2)=\langle (D_m\Phi_{\lambda_2})^2,(D_n\Phi_{\lambda_2})^2\rangle\\
=&\int_{\R^{4d}}\partial_{x_m}G_{\lambda_1}(y_1)\partial_{x_m}G_{\lambda_1}(z_1)\partial_{x_n}G_{\lambda_2}(y_2)\partial_{x_n}G_{\lambda_2}(z_2)\E\{V(y_1)V(y_2)V(z_1)V(z_2)\}dy_1dy_2dz_1dz_2.
\end{aligned}
\end{equation}
Clearly $\partial_{x_k}G_\lambda(y)\to \partial_{x_k}G_0(y)$ almost everywhere as $\lambda\to 0$ with $G_0$ the Green's function of $-\frac12\Delta$. We also have the bound $|\nabla G_\lambda(y)|\les |y|^{1-d}$. Moreover, by the strongly mixing property in Assumption \ref{ass:mixing} and \cite[Lemma 2.3]{hairer2013random}, we have 
\begin{equation}
|\E\{V(y_1)V(y_2)V(z_1)V(z_2)\}|\les \Psi(y_1-y_2)\Psi(z_1-z_2)+\Psi(y_1-z_2)\Psi(z_1-y_2)+\Psi(y_1-z_1)\Psi(y_2-z_2)
\end{equation}
for some $\Psi$ satisfying $|\Psi(x)|\les 1\wedge |x|^{-\beta}$ for any $\beta>0$. By the dominated convergence theorem and the convergence of $\sigma_\lambda^2\to \sigma^2$, we have the convergence of $\sum_{m,n=1}^dI_{mn}(\lambda_1,\lambda_2)-\sigma_{\lambda_1}^2\sigma_{\lambda_2}^2$. So $\sum_{k=1}^d (D_k\Phi_\lambda)^2-\sigma_\lambda^2\to \sum_{k=1}^d (D_k\Phi)^2-\sigma^2$ in $L^2(\Omega)$.

For the uniform boundedness of $\eps\int_0^{t/\eps^2}Z_\xi(y_s)ds$, by \cite[Lemma 3.4]{gu2013weak}, we only need to show the integrability of $R_{1,\xi}(x)|x|^{2-d}$ and $R_{2,\xi}(x)|x|^{2-d}$, with $R_{1,\xi},R_{2,\xi}$ the covariance function of $2\sum_{k=1}^d\xi_kD_k\Phi$ and $\sum_{k=1}^d (D_k\Phi)^2-\sigma^2$ respectively. By the convergence in $L^2(\Omega)$, $R_{i,\xi}(x)=\lim_{\lambda\to 0}R_{i,\lambda,\xi}(x), i=1,2$, where $R_{1,\lambda,\xi},R_{2,\lambda,\xi}$ are the covariance function of $2\sum_{k=1}^d\xi_kD_k\Phi_\lambda$ and $\sum_{k=1}^d (D_k\Phi_\lambda)^2-\sigma_\lambda^2$. \cite[Proposition 3.5]{gu2013weak} shows that
\begin{equation}
|R_{1,\lambda,\xi}(x)|+|R_{2,\lambda,\xi}(x)|\les (1+|\xi|)^2 \left( \lambda^{\frac{d}{2}-1}e^{-c\sqrt{\lambda}|x|}+1\wedge \frac{e^{-c\sqrt{\lambda}|x|}}{|x|^{d-2}}+1\wedge \frac{1}{|x|^\beta}\right)
\end{equation}
for some $c>0$ and $\beta>0$ sufficiently large. By taking the limit $\lambda\to 0$, we obtain
\begin{equation}
|R_{1,\xi}(x)|+|R_{2,\xi}(x)|\les (1+|\xi|)^2 \left(1\wedge \frac{1}{|x|^{d-2}}+1\wedge \frac{1}{|x|^\beta}\right),
\end{equation}
and clearly it implies the integrability of $(|R_{1,\xi}(x)|+|R_{2,\xi}(x)|)|x|^{2-d}$ since $d\geq 5$. The proof is complete.
\end{proof}

Now we show that for $\eps\int_0^{t/\eps^2}Z_\xi(y_s)ds$, we can apply Kipnis-Varadhan's result. Since we are in the probability space $\Omega$ with the measure-preserving and ergodic transformations $\{\tau_x,x\in \R^d\}$, the only assumption we need to verify is $\langle Z_\xi,-L^{-1}Z_\xi\rangle<\infty$, see \cite[Assumption 2.1]{gu2013weak} and the proof of \cite[Theorem 2.2]{gu2013weak}. By Kipnis-Varadhan, it is equivalent with the finiteness of the asymptotic variance, i.e., $\E\E_B\{(\eps\int_0^{t/\eps^2}Z_\xi(y_s)ds)^2\}$ is bounded uniformly in $\eps$ in our context. For the sake of convenience, we present the proof in the following lemma.

\begin{lemma}
For any $\V\in L^2(\Omega)$, $\langle \V,-L^{-1}\V\rangle<\infty$ is equivalent with the fact that $\E\E_B\{(\frac{1}{\sqrt{t}}\int_0^{t}\V(\tau_{B_s}\omega)ds)^2\}$ is bounded uniformly in $t$.
\end{lemma}

\begin{proof}
First, let $\hat{R}(\xi)$ be the power spectrum associated with $\V$, then $\langle \V,-L^{-1}\V\rangle<\infty$ is equivalent with the integrability of $\hat{R}(\xi)|\xi|^{-2}$. Now by an explicit calculation, we have
\begin{equation}
\E\E_B\{(\frac{1}{\sqrt{t}}\int_0^{t}\V(\tau_{B_s}\omega)ds)^2\}=\frac{2}{t}\int_0^t\left(\int_0^s\int_{\R^d}\frac{1}{(2\pi)^d}\hat{R}(\xi)e^{-\frac12|\xi|^2u}d\xi du\right)ds.
\end{equation}
As a function of $s$, $\int_0^s\int_{\R^d}\frac{1}{(2\pi)^d}\hat{R}(\xi)e^{-\frac12|\xi|^2u}d\xi du$ is positive and increasing, so for the r.h.s. of the above display to be uniformly bounded in $t$, an equivalent condition is $$\lim_{s\to\infty}\int_0^s\int_{\R^d}\frac{1}{(2\pi)^d}\hat{R}(\xi)e^{-\frac12|\xi|^2u}d\xi du<\infty,$$ i.e., the integrability of $\hat{R}(\xi)|\xi|^{-2}$. The proof is complete.
\end{proof}

By Lemma \ref{lem:martingaleT1d5}, we know that
\begin{equation}
\frac{v_{2,\eps}}{\eps}\approx 
-\frac12\int_{\R^d}\frac{1}{(2\pi)^d}\hat{f}(\xi)e^{i\xi\cdot x}\E_B\{e^{i\tilde{M}_t^\eps}\eps\int_0^{t/\eps^2}Z_{\xi}(y_s)ds\}d\xi,
\end{equation}
where $\approx$ means the error goes to zero in $L^1(\Omega)$. Since $\tilde{M}_t^\eps=\sum_{k=1}^d\eps\int_0^{t/\eps^2}(\xi_k+D_k\Phi_\lambda(y_s))dB_s^k$, by the convergence $D_k\Phi_\lambda\to D_k\Phi$ in $L^2(\Omega)$, we further obtain
\begin{equation}
\frac{v_{2,\eps}}{\eps}\approx 
-\frac12\int_{\R^d}\frac{1}{(2\pi)^d}\hat{f}(\xi)e^{i\xi\cdot x}\E_B\{e^{i\sum_{k=1}^d\eps\int_0^{t/\eps^2}(\xi_k+D_k\Phi(y_s))dB_s^k}\eps\int_0^{t/\eps^2}Z_{\xi}(y_s)ds\}d\xi.
\end{equation}

By Kipnis-Varadhan's method and Proposition \ref{prop:martingaleT2d5}, 
\begin{equation}
\E\E_B\{\left(\eps\int_0^{t/\eps^2}Z_{\xi}(y_s)ds-\sum_{k=1}^d \eps\int_0^{t/\eps^2}D_k\tilde{\Phi}(y_s)dB_s^k\right)^2\}\to 0,
\end{equation} 
where $\tilde{\Phi}$ is the corrector corresponding to $Z_\xi$. This leads to
\begin{equation}
\frac{v_{2,\eps}}{\eps}\approx 
-\frac12\int_{\R^d}\frac{1}{(2\pi)^d}\hat{f}(\xi)e^{i\xi\cdot x}\E_B\{e^{i\sum_{k=1}^d\eps\int_0^{t/\eps^2}(\xi_k+D_k\Phi(y_s))dB_s^k}\sum_{k=1}^d \eps\int_0^{t/\eps^2}D_k\tilde{\Phi}(y_s)dB_s^k\}d\xi.
\end{equation}

Since we have two martingales here, we apply martingale central limit theorem and ergodic theorem to show that $\E_B\{e^{i\sum_{k=1}^d\eps\int_0^{t/\eps^2}(\xi_k+D_k\Phi(y_s))dB_s^k}\sum_{k=1}^d \eps\int_0^{t/\eps^2}D_k\tilde{\Phi}(y_s)dB_s^k\}$ converges to some constant in $L^1(\Omega)$.

\begin{proposition}
Define $\Sigma_{11}=|\xi|^2+\sum_{k=1}^d\|D_k\Phi\|^2$ and $\Sigma_{12}=\sum_{k=1}^d \langle D_k\Phi,D_k\tilde{\Phi}\rangle$, we have 
\begin{equation}
\E_B\{e^{i\sum_{k=1}^d\eps\int_0^{t/\eps^2}(\xi_k+D_k\Phi(y_s))dB_s^k}\sum_{k=1}^d \eps\int_0^{t/\eps^2}D_k\tilde{\Phi}(y_s)dB_s^k\}\to ie^{-\frac12\Sigma_{11}t}\Sigma_{12}t
\end{equation}
 in $L^1(\Omega)$.
\label{prop:martingaleT3d5}
\end{proposition}

\begin{proof}
By stationarity, we can replace $y_s$ by $\tau_{B_s}\omega$. To simplify the notation, we let $N_1^\eps=\sum_{k=1}^d\eps\int_0^{t/\eps^2}(\xi_k+D_k\Phi(y_s))dB_s^k$ and $N_2^\eps=\sum_{k=1}^d \eps\int_0^{t/\eps^2}D_k\tilde{\Phi}(y_s)dB_s^k$. By martingale central limit theorem and ergodic theorem, we obtain that for almost every $\omega\in \Omega$, 
\begin{equation}
(N_1^\eps,N_2^\eps)\Rightarrow (N_1,N_2),
\end{equation}
where $N_1,N_2$ are joint Gaussian with mean zero and covariance matrix $t\Sigma $, and $\Sigma_{11}=|\xi|^2+\sum_{k=1}^d\|D_k\Phi\|^2$, $\Sigma_{12}=\sum_{k=1}^d \langle D_k\Phi,D_k\tilde{\Phi}\rangle$, and $\Sigma_{22}=\sum_{k=1}^d \|D_k\tilde{\Phi}\|^2$. Let $g_K(x)=(x\wedge K)\vee (-K)$ be a continuous and bounded cutoff function for $K>0$, and $h_K(x)=x-g_K(x)$ we have 
\begin{equation}
\E_B\{e^{iN_1^\eps}N_2^\eps\}=\E_B\{e^{iN_1^\eps}g_K(N_2^\eps)\}+\E_B\{e^{iN_1^\eps}h_K(N_2^\eps)\}
\end{equation}
For the second term, we have 
\begin{equation}
\E\E_B\{|h_K(N_2^\eps)|\}\leq \E\E_B\{|N_2^\eps|1_{|N_2^\eps|\geq K}\}\leq \frac{1}{K}\E\E_B\{|N_2^\eps|^2\}\les \frac{1}{K}.
\end{equation}
Therefore, 
\begin{equation}
\begin{aligned}
\limsup_{\eps\to 0}\E\{|\E_B\{e^{iN_1^\eps}N_2^\eps\}-\E\{e^{iN_1}N_2\}|\}\les& \lim_{\eps\to 0}\E\{|\E_B\{e^{iN_1^\eps}g_K(N_2^\eps)\}-\E\{e^{iN_1}g_K(N_2)\}|\}\\
&+|\E\{e^{iN_1}N_2\}-\E\{e^{iN_1}g_K(N_2)\}|+\frac{1}{K}\\
=&|\E\{e^{iN_1}N_2\}-\E\{e^{iN_1}g_K(N_2)\}|+\frac{1}{K}
\end{aligned}
\end{equation}
by the weak convergence and dominated convergence theorem. Now we let $K\to \infty$ and calculate $\E\{e^{iN_1}N_2\}=ie^{-\frac12\Sigma_{11}t}\Sigma_{12}t$ to complete the proof.
\end{proof}

The above proposition implies that 
\begin{equation}
\begin{aligned}
\frac{v_{2,\eps}}{\eps}\to &
-\frac12\int_{\R^d}\frac{1}{(2\pi)^d}\hat{f}(\xi)e^{i\xi\cdot x}ie^{-\frac12(|\xi|^2+\sigma^2)t}d\xi\sum_{k=1}^d \langle D_k\Phi,D_k\tilde{\Phi}\rangle t\\
=&-\frac12iu_{hom}(t,x)\sum_{k=1}^d \langle D_k\Phi,D_k\tilde{\Phi}\rangle t
\end{aligned}
\label{eq:EsV2d5}
\end{equation}
in $L^1(\Omega)$. By combining \eqref{eq:decom1Cord5}, \eqref{eq:EsV1d5} and \eqref{eq:EsV2d5}, we conclude that
\begin{equation}
u_\eps(t,x)-u_{hom}(t,x)=v_{1,\eps}+v_{2,\eps}+o(\eps)=iu_{hom}(t,x)\Phi(\tau_{\frac{x}{\eps}}\omega)+iu_{hom}(t,x)C_\V t+o(\eps)
\end{equation}
if we define 
\begin{equation}
C_\V:=-\frac12\sum_{k=1}^d \langle D_k\Phi,D_k\tilde{\Phi}\rangle.
\label{eq:biasCV}
\end{equation}
The proof of Theorem \ref{thm:THd5} is complete.

If we assume the symmetry condition of the distribution of $V(x)=\V(\tau_x\omega)$, $\E\{V(x_1)V(x_2)V(x_3)\}=0,\forall x_1,x_2,x_3\in \R^d$ as in the Gaussian case, by a direct calculation we can obtain $C_\V=0$, i.e., the bias vanishes.

\subsection{Analysis of $v_{1,\eps}$: $d=4$}

When $d=4$, by assuming the initial condition $f\equiv 1$, the error from the martingale part is negligible. From \eqref{eq:decom1Cord4}, we have
\begin{equation}
\frac{u_\eps(t,x)-u_{hom}(t,x)}{\eps|\log\eps|^\frac12}\approx \frac{v_{1,\eps}}{\eps|\log\eps|^\frac12}=\frac{\E_B\{iR_t^\eps e^{iM_t^\eps}\}}{\eps|\log\eps|^\frac12},
\end{equation}
with the error going to zero in $L^1(\Omega)$. We recall that
$R^\eps_t=\eps\int_0^{t/\eps^2}\lambda\Phi_\lambda(y_s)ds-\eps\Phi_\lambda(y_{t/\eps^2})+\eps\Phi_\lambda(y_0)$
with $y_s=\tau_{\frac{x}{\eps}+B_s}\omega$. The analysis is very similar with $d\geq 5$, and we present it through the following lemmas in parallel with Lemma \ref{lem:remainderT1d5}, \ref{lem:remainderT2d5}, \ref{lem:remainderT3d5}.

\begin{lemma}
$|\log\eps|^{-\frac12}\E_B\{e^{iM_t^\eps}\Phi_\lambda(y_0)\}-e^{-\frac12\sigma^2t}|\log\eps|^{-\frac12}\Phi_\lambda(y_0)\to 0$ in $L^1(\Omega)$ as $\eps \to 0$.
\label{lem:remainderT1d4}
\end{lemma}

\begin{proof}
The proof is the same as Lemma \ref{lem:remainderT1d5} if we note that $|\log\eps|^{-1}\langle\Phi_\lambda,\Phi_\lambda\rangle<\infty$ and $\E\E_B\{|\langle M^\eps\rangle_t-\sigma_\lambda^2t|^2\}+|\sigma_\lambda^2-\sigma^2|\les\eps^2|\log\eps|$.
\end{proof}

\begin{lemma}
$|\log\eps|^{-\frac12}\E_B\{e^{iM_t^\eps}\Phi_\lambda(y_{t/\eps^2})\}\to0 $ in $L^1(\Omega)$ as $\eps\to 0$.
\label{lem:remainderT2d4}
\end{lemma}

\begin{proof}
The proof is the same as Lemma \ref{lem:remainderT2d5} except that we have to show as in \eqref{eq:proofremainderT2d5} that
\begin{equation}
\frac{1}{|\log\eps|}\E\{|\E_B\{\Phi_\lambda(y_{s})\}|^2\}=\frac{1}{|\log\eps|}\int_{\R^d}\frac{1}{(2\pi)^d}\frac{\hat{R}(\xi)}{(\lambda+\frac12|\xi|^2)^2}e^{-|\xi|^2s}\to 0
\end{equation}
for $s\in (0,t/\eps^2)$ chosen so that $\eps^2s\to 0$. By Lemma \ref{lem:proofLem711}, we have 
\begin{equation}
\frac{1}{|\log\eps|}\int_{\R^d}\frac{1}{(2\pi)^d}\frac{\hat{R}(\xi)}{(\lambda+\frac12|\xi|^2)^2}e^{-|\xi|^2s}d\xi=\frac{1}{|\log\eps|}\int_{\R^d}\frac{1}{(2\pi)^d}\frac{\hat{R}(\sqrt{\lambda}\xi)}{(1+\frac12|\xi|^2)^2}e^{-|\xi|^2\lambda s}d\xi\les \frac{1+|\log \lambda s|}{|\log\eps|}.
\end{equation}
Now we choose $s=\eps^{-2}|\log\eps|^{-1}$. In this way $|\log \lambda s|=\log|\log\eps|\ll |\log\eps|$ and $
\eps^2s=|\log\eps|^{-1}\to 0$ as $\eps \to 0$. The proof is complete.
\end{proof}

\begin{lemma}
$|\log\eps|^{-\frac12}\E_B\{e^{iM_t^\eps}\lambda\int_0^{t/\eps^2}\Phi_\lambda(y_s)ds
\}\to0 $ in $L^2(\Omega)$ as $\eps\to 0$.
\label{lem:remainderT3d4}
\end{lemma}

\begin{proof}
By an explicit calculation and Lemma \ref{lem:proofLem711}, 
\begin{equation}
\begin{aligned}
\E\E_B\{|\lambda\int_0^{t/\eps^2}\Phi_\lambda(y_s)ds|^2\}=&\frac{1}{(2\pi)^d}\int_0^{t/\eps^2}\int_0^{t/\eps^2}\int_{\R^d}\lambda^2\frac{\hat{R}(\xi)}{(\lambda+\frac12|\xi|^2)^2}e^{-\frac12|\xi|^2|s-u|}d\xi dsdu\\
\les &\int_0^t\int_0^t\int_{\R^d}\frac{1}{(1+\frac12|\xi|^2)^2}e^{-\frac12|\xi|^2|s-u|}d\xi dsdu\\
\les &\int_0^t\int_0^t(1+\log |s-u|)dsdu<\infty,
\end{aligned}
\end{equation}
so the proof is complete.
\end{proof}

Now we can combine \eqref{eq:decom1Cord4}, Lemma \ref{lem:remainderT1d4}, \ref{lem:remainderT2d4} and \ref{lem:remainderT3d4} to conclude that
\begin{equation}
\frac{u_\eps(t,x)-u_{hom}(t,x)}{\eps|\log\eps|^\frac12}-iu_{hom}(t,x)\frac{\Phi_{\eps^2}(\tau_{\frac{x}{\eps}}\omega)}{|\log\eps|^\frac12}\to 0
\end{equation}
in $L^1(\Omega)$. The proof of Theorem \ref{thm:THd4} is complete.

\subsection{Proof of Corollary \ref{cor:THd4}: $d=4$}

The last goal is to prove the convergence in distribution of $|\log\eps|^{-\frac12}\Phi_{\eps^2}(\tau_x\omega)$ for Gaussian and Poissonian potentials.  To keep the notation simple, we consider $|\log\lambda|^{-\frac12}\Phi_\lambda(\tau_x\omega)$.

For the Gaussian case, since $\E\{\Phi_\lambda\}=0$, we only need to show the convergence of variance, i.e., $|\log\lambda|^{-1}\langle\Phi_\lambda,\Phi_\lambda\rangle$. This is given by Lemma \ref{lem:conVard4}.

For the Poissonian case, we prove the weak convergence again by the characteristic function. 

\begin{lemma}
If $V$ is Poissonian as in Assumption \ref{ass:potential}, then we have $\forall \theta\in \R$:
\begin{equation}
\E\{\exp(i\theta \frac{\Phi_{\lambda}(\tau_x\omega)}{|\log\lambda|^\frac12})\}\to \exp(-\frac{\hat{R}(0)\theta^2}{(2\pi)^d}).
\end{equation}
\end{lemma}

\begin{proof}
The proof is similar to Lemma \ref{lem:conChr}.

By stationarity, we can choose $x=0$. For Poissonian potential $V(x)=\int_{\R^d}\varphi(x-y)\omega(dy)-c_\varphi$, then $\Phi_{\lambda}(\tau_0\omega)=\int_{\R^d}f^\lambda(-x)\omega(dx)-c_\varphi\int_{\R^d}G_\lambda(x)dx$, where $f^\lambda(x)=\int_{\R^d}G_\lambda(x-y)\varphi(y)dy$.

Now we can write
\begin{equation}
\E\{\exp(i\theta \frac{\Phi_{\lambda}(\tau_0\omega)}{|\log\lambda|^\frac12})\}=\exp\left(\int_{\R^d}(e^{i\theta|\log\lambda|^{-\frac12}f_\lambda(-x)}-1)dx-i\theta|\log\lambda|^{-\frac12}c_\varphi\int_{\R^d}G_\lambda(x)dx\right).
\end{equation}
Since $\int_{\R^d}f_\lambda(-x)dx=c_\varphi\int_{\R^d}G_\lambda(x)dx$, after a Taylor expansion and change of variables, we have
\begin{equation}
\E\{\exp(i\theta \frac{\Phi_{\lambda}(\tau_x\omega)}{|\log\lambda|^\frac12})\}=\exp\left(\int_{\R^d}\sum_{k=2}^\infty \frac{1}{k!} (i\theta|\log\lambda|^{-\frac12}f_\lambda(x))^kdx\right).
\end{equation}

First, when $k=2$, we have $\int_{\R^d}f_\lambda(x)^2dx=\langle\Phi_\lambda,\Phi_\lambda\rangle$, therefore by Lemma \ref{lem:conVard4}, 
\begin{equation}
\frac{\int_{\R^d}f_\lambda(x)^2dx}{|\log\lambda|}\to \frac{2\hat{R}(0)}{(2\pi)^d}.
\end{equation}

For $k\geq 3$, we have
\begin{equation}
\begin{aligned}
\frac{\int_{\R^d}|f_\lambda(x)|^kdx}{|\log\lambda|^{\frac{k}{2}}}\leq & C^k\frac{1}{|\log\lambda|^{\frac{k}{2}}}\int_{\R^{kd}}\prod_{i=1}^k \frac{|\F\{|\varphi|\}(\xi_i)|}{\lambda+\frac12|\xi_i|^2}\delta(\sum_{i=1}^k\xi_i)d\xi\\
=&C^k\frac{1}{|\log\lambda|^{\frac{k}{2}}}\int_{\R^{(k-1)d}}\prod_{i=1}^{k-1} \frac{|\F\{|\varphi|\}(\xi_i)|}{\lambda+\frac12|\xi_i|^2}\frac{|\F\{|\varphi|\}(-\xi_1-\ldots-\xi_{k-1})|}{\lambda+\frac12|\xi_1+\ldots+\xi_{k-1}|^2}d\xi
\end{aligned}
\end{equation}
for some constant $C$. Since $\varphi$ is continuous and compactly supported, we can assume here $|\F\{|\varphi|\}|$ to be bounded, fast-decaying, radially symmetric and decreasing. Then for the integration in $\xi_{k-1}$, by Lemma \ref{lem:MaxOrigin}, we obtain
\begin{equation}
\int_{\R^d}\frac{|\F\{|\varphi|\}(\xi_{k-1})|}{\lambda+\frac12|\xi_{k-1}|^2}
\frac{|\F\{|\varphi|\}(-\xi_1-\ldots-\xi_{k-1})|}{\lambda+\frac12|\xi_1+\ldots+\xi_{k-1}|^2}d\xi_{k-1}\leq \int_{\R^d}\left(\frac{\F\{|\varphi|\}(\xi_{k-1})}{\lambda+\frac12|\xi_{k-1}|^2}\right)^2d\xi_{k-1}.
\end{equation}
The r.h.s. of the above display is of the form $\langle\Phi_\lambda,\Phi_\lambda\rangle$ with $\hat{R}(\xi)$ replaced by $|\F\{|\varphi|\}(\xi_{k-1})|^2$, so by Lemma \ref{lem:conVard4}, we obtain
\begin{equation}
\frac{\int_{\R^d}|f_\lambda(x)|^kdx}{|\log\lambda|^{\frac{k}{2}}}\leq C^k\frac{1}{|\log \lambda|^{\frac{k-2}{2}}}\int_{\R^{(k-2)d}}\prod_{i=1}^{k-2} \frac{\F\{|\varphi|\}(\xi_i)}{\lambda+\frac12|\xi_i|^2}d\xi\leq \frac{C^k}{|\log\lambda|^{\frac{k-2}{2}}}
\end{equation}
for some possibly different constant $C>0$. This leads to
\begin{equation}
\int_{\R^d}\sum_{k\geq 3}\frac{1}{k!}(|\theta||\log\lambda|^{-\frac12}|f_\lambda(x)|)^kdx\leq \sum_{k\geq 3}\frac{C^k|\theta|^k}{k!|\log\lambda|^{\frac{k-2}{2}}}\to 0
\end{equation}
as $\lambda\to 0$. The proof is complete.
\end{proof}

\section*{Acknowledgment}  
The authors would like to thank the anonymous referees for their careful reading of the manuscript. The first author would like to thank Jean-Christophe Mourrat for many helpful discussions, in particular for sharing the ideas of proving Lemma \ref{lem:remainderT2d5} and applying Kipnis-Varadhan's method to the difference between quadratic variations in Section \ref{sec:conMou}. This paper was partially funded by AFOSR Grant NSSEFF-FA9550-10-1-0194 and NSF grant DMS-1108608.

\appendix
\section{Technical lemmas}
\label{sec:teLem}

\begin{lemma}
If $h_1\in L^1,h_2,h_3\in L^1\cap L^2$
, then
\begin{equation}
\begin{aligned}
&\E\{e^{i\int_{\R^d}h_1(y)\omega(dy)}\int_{\R^d}h_2(y)\omega(dy)\int_{\R^d}h_3(y)\omega(dy)\}\\
=&\exp(\int_{\R^d}(e^{ih_1(y)}-1)dy)\left(\int_{\R^d}e^{ih_1(y)}h_2(y)h_3(y)dy+\int_{\R^d}e^{ih_1(y)}h_2(y)dy\int_{\R^d}e^{ih_1(y)}h_3(y)dy\right).
\end{aligned}
\end{equation}
\label{lem:PoissonMoment}
\end{lemma}

\begin{proof}
If $h_i$ are all compactly supported, it is a direct calculation. The general case can be proved by approximation.
\end{proof}

\begin{lemma}
\begin{eqnarray}
\exp(\int_{\R^d}(e^{i\eps\int_0^{t/\eps^2}\varphi(B_s-y)ds}-1-i\frac{c_\varphi t}{\eps})dy)&\to& e^{-\frac12\sigma^2 t},
\label{eq:singleWKcon}\\
\exp(\int_{\R^d}(e^{i\eps\int_0^{t/\eps^2}(\varphi(B_s-y)-\varphi(W_s-y))ds}-1)dy)&\to& e^{-\sigma^2 t}
\label{eq:doubleWKcon}
\end{eqnarray}
in probability as $\eps \to 0$.
\label{lem:conPoissonChr}
\end{lemma}

\begin{proof}
We first point out that \eqref{eq:singleWKcon} and \eqref{eq:doubleWKcon} is related to the convergence of the annealed characteristic function for random variables of the form $\eps\int_0^{t/\eps^2} V(B_s)ds$ and $\eps\int_0^{t/\eps^2}(V(B_s)-V(W_s))ds$, respectively, when $V$ is Poissonian. This is discussed in detail in \cite{gu2013invariance}. By \cite[Proposition 3.7, 3.8]{gu2013invariance}, we have that as $\eps\to 0$
\begin{eqnarray}
\int_{\R^d}\left(\eps\int_0^{t/\eps^2}\varphi(B_s-y)ds\right)^2dy\to \sigma^2 t,\\
\int_{\R^d}\sum_{k\geq 3}\frac{1}{k!}\left(\eps\int_0^{t/\eps^2}|\varphi|(B_s-y)ds\right)^kdy\to 0
\label{eq:conGeq3}
\end{eqnarray}
in probability, which directly leads to \eqref{eq:singleWKcon} if we expand $e^{ix}$ in power series and use the fact $\int_{\R^d}\varphi(x)dx=c_\varphi$.

For \eqref{eq:doubleWKcon}, we use the fact that $|a+b|^k\leq 2^{k-1}(|a|^k+|b|^k)$ together with \eqref{eq:conGeq3} to derive that
\begin{equation*}
\int_{\R^d}\sum_{k\geq 3}\frac{1}{k!}\left(i\eps\int_0^{t/\eps^2}(\varphi(B_s-y)-\varphi(W_s-y))ds\right)^kdy\to 0
\end{equation*}
in probability. The rest is to show 
\begin{equation*}\eps^2\int_0^{t/\eps^2}\int_0^{t/\eps^2}\varphi(B_s-y)\varphi(W_u-y)dsdu=\eps^2\int_0^{t/\eps^2}\int_0^{t/\eps^2}R(B_s-W_u)dsdu\to 0
\end{equation*} in probability. Assuming $R$ is positive without loss of generality, we have in Fourier domain that
\begin{equation*}
\begin{aligned}
\E\{\eps^2\int_0^{t/\eps^2}\int_0^{t/\eps^2}R(B_s-W_u)dsdu\}\les &\eps^2\int_0^{t/\eps^2}\int_0^{t/\eps^2}\int_{\R^d}\hat{R}(\xi)e^{-\frac12|\xi|^2(s+u)}d\xi dsdu\\
\les& \int_{\R^d}\frac{\hat{R}(\xi)}{|\xi|^2}\frac{\eps^2}{|\xi|^2}(1-e^{-\frac12|\xi|^2\frac{t}{\eps^2}})d\xi\to 0
\end{aligned}
\end{equation*}
by the dominated convergence theorem, which completes the proof.
\end{proof}

\begin{lemma}
\begin{eqnarray}
\eps^2\int_0^{t/\eps^2}\int_0^{t/\eps^2}R(B_s-B_u)dsdu&\to& \sigma^2 t\label{eq:conVarGauss}\\
\eps^2\int_0^{t/\eps^2}\int_0^{t/\eps^2}R(B_s-W_u)dsdu&\to& 0
\label{eq:conCrossGauss}
\end{eqnarray}
 in probability as $\eps\to 0$.
\label{lem:conGaussChr}
\end{lemma}

\begin{proof}
We note that $\eps^2\int_0^{t/\eps^2}\int_0^{t/\eps^2}R(B_s-B_u)dsdu$ is the variance of $\eps\int_0^{t/\eps^2}V(B_s)ds$. By \cite[Proposition 3.7]{gu2013invariance}, we obtain \eqref{eq:conVarGauss}. The proof of \eqref{eq:conCrossGauss} is contained in the proof of Lemma \ref{lem:conPoissonChr}.
\end{proof}

\begin{lemma}
Assume $f$ is positive, radially symmetric and decreasing, and integrable around the origin, $g$ is bounded, integrable, positive, radially symmetric and decreasing, the $f\star g$ is bounded, radially symmetric and decreasing.
\label{lem:MaxOrigin}
\end{lemma}

\begin{proof}
Clearly, $f\star g$ is bounded and radially symmetric, so we only need to prove it is radially decreasing.

By Fubini theorem and symmetry, it can be reduced to the one-dimensional case. Let $F(x)=\int_\R f(x+y)g(y)dy$, and by approximation, we assume $f$ is smooth and bounded. So $F'(x)=\int_\R f'(x+y)g(y)dy=\int_\R f'(y)g(y-x)dy$, which implies
\begin{equation*}
F'(x)=\int_{-\infty}^0f'(y)(g(x-y)-g(x+y))dy.
\end{equation*}
When $x>0,y<0$, we have $x-y=|x|+|y|\geq |x+y|$, so $g(x-y)\leq g(x+y)$, and since $f'(y)\geq 0$, we have $F'(x)\leq 0$. The proof is complete.
\end{proof}

\begin{lemma}
\begin{eqnarray}
\sup_{y\in \R^d}\int_{\R^d}\frac{|f^\lambda|(x+y)}{|x|^{d-2}}dx&\les& \eps^{-1},\\
\sup_{y\in \R^d}\int_{\R^d}\frac{|f_k^\lambda|(x+y)}{|x|^{d-2}}dx&\les& |\log\eps|,\\
\sup_{y\in \R^d}\int_{\R^d}\frac{(|f^\lambda|\star|\varphi|)(x+y)}{|x|^{d-2}}dx&\les& \eps^{-1},\\
\sup_{y\in \R^d}\int_{\R^d}\frac{(|f_k^\lambda|\star|\varphi|)(x+y)}{|x|^{d-2}}dx&\les& |\log\eps|.
\end{eqnarray}
\label{lem:boundConvolution}
\end{lemma}
\begin{proof}
Recall that $f^\lambda=\varphi\star G_\lambda, f_k^\lambda=\varphi\star \partial_{x_k}G_\lambda$, and $G_\lambda$ is the Green's function of $\lambda-\frac12\Delta$, so $|G_\lambda(x)|\les e^{-c\sqrt{\lambda}|x|}|x|^{2-d}$, $|\partial_{x_k}G_\lambda(x)|\les e^{-c\sqrt{\lambda}|x|}|x|^{1-d}$ for some constant $c>0$. Without loss of generality, we assume $|\varphi|$ is bounded, radially symmetry and decreasing function with compact support, and replace $G_\lambda, \partial_{x_k}G_\lambda$ by the above bounds in the estimates.

We take $\int_{\R^d}|f^\lambda|(x+y)|x|^{2-d}dx$ for example. The proof of the other inequalities is similar.
\begin{equation*}
\begin{aligned}
\int_{\R^d}\frac{|f^\lambda(x+y)|}{|x|^{d-2}}dx\les& \int_{\R^d}\frac{1}{|x|^{d-2}}\int_{\R^d}|\varphi|(x+y-z)\frac{e^{-c\sqrt{\lambda}|z|}}{|z|^{d-2}}dzdx\\
\leq &\int_{\R^{2d}}\frac{|\varphi|(x-z)}{|x|^{d-2}}\frac{e^{-c\sqrt{\lambda}|z|}}{|z|^{d-2}}dzdx
\end{aligned}
\end{equation*}
by Lemma \ref{lem:MaxOrigin}, since $|\varphi|\star(e^{-c\sqrt{\lambda}|x|}|x|^{2-d})$ is a bounded, integrable, radially symmetric and decreasing function again by Lemma \ref{lem:MaxOrigin}. Now we assume $|\varphi|(x)\les 1\wedge |x|^{-\alpha}$ for some $\alpha>0$ sufficiently large, and bound the integral in $z$ by
\begin{equation*}
\int_{\R^d}\frac{e^{-c\sqrt{\lambda}|z|}}{|z|^{d-2}} \left(1\wedge \frac{1}{|x-z|^\alpha}\right)dz\les 1\wedge \left(\frac{1}{|x|^{\alpha-2}}+\frac{e^{-\rho \sqrt{\lambda}|x|}}{|x|^{d-2}}\right)
\end{equation*}
for some constant $\rho>0$ \cite[Lemma A.3]{gu2013weak}.
 The rest is a straightforward calculation.
\end{proof}

\begin{lemma}
Assume $f,g$ are bounded, integrable, positive, and radially symmetric and decreasing, then we have
\begin{equation}
\E_B\E_W\int_{[0,t/\eps^2]^2}f(x_s-\tilde{B})g(y_u-\tilde{W})dsdu\les \sup_{y\in \R^d}\int_{\R^d}\frac{f(x+y)}{|x|^{d-2}}dx \sup_{y\in \R^d}\int_{\R^d}\frac{g(x+y)}{|x|^{d-2}}dx
\end{equation}
where $x,y\in \{B,W\}$, and $\tilde{B},\tilde{W} \in\{0,B_{t/\eps^2},W_{t/\eps^2}\}$.
\label{lem:Brown2ndMo}
\end{lemma}

\begin{proof}
The proofs for all choices of $x_s,y_u,\tilde{B},\tilde{W}$ are similar. We take one example that contains all the ingredients.

Let $x_s=B_s, y_u=B_u, \tilde{B}=B_{t/\eps^2}, \tilde{W}=W_{t/\eps^2}$, and we consider
\begin{equation*}
\E_B\E_W\int_{[0,t/\eps^2]^2}f(B_s-B_{t/\eps^2})g(B_u-W_{t/\eps^2})dsdu=(i)+(ii),
\end{equation*}
where \begin{eqnarray*}
(i)&=&\E_B\E_W\int_{0<s<u<t/\eps^2}f(B_s-B_{t/\eps^2})g(B_u-W_{t/\eps^2})dsdu,\\
(ii)&=&\E_B\E_W\int_{0<u<s<t/\eps^2}f(B_s-B_{t/\eps^2})g(B_u-W_{t/\eps^2})dsdu.
\end{eqnarray*}

For $(i)$, by change of variables, we have
\begin{equation*}
(i)=\E_W\int_{\R_+^2}\int_{\R^{3d}}1_{u_1+u_2< t/\eps^2}f(x+y)g(z+x-W_{t/\eps^2})q_{u_1}(z)q_{u_2}(y)q_{t/\eps^2-u_1-u_2}(x)dxdydzdu_1du_2.
\end{equation*}
For the integrals in $y,z$, by Lemma \ref{lem:MaxOrigin}, we have
\begin{equation*}
\begin{aligned}
(i)\leq& \E_W\int_{\R_+^2}\int_{\R^{3d}}1_{u_1+u_2< t/\eps^2}f(y)g(z)q_{u_1}(z)q_{u_2}(y)q_{t/\eps^2-u_1-u_2}(x)dxdydzdu_1du_2\\
=&\int_{\R_+^2}\int_{\R^{3d}}1_{u_1+u_2< t/\eps^2}f(y)g(z)q_{u_1}(z)q_{u_2}(y)dydzdu_1du_2.
\end{aligned}
\end{equation*}
By change of variables $\lambda_1=-\frac{|z|^2}{2u_1}, \lambda_2=-\frac{|y|^2}{2u_2}$, we have $$(i)\les \sup_{y\in \R^d}\int_{\R^d}\frac{f(x+y)}{|x|^{d-2}}dx \sup_{y\in \R^d}\int_{\R^d}\frac{g(x+y)}{|x|^{d-2}}dx.$$

For $(ii)$, by change of variables, we have
\begin{equation*}
(ii)=\E_W\int_{\R_+^2}\int_{\R^{3d}}1_{u_1+u_2< t/\eps^2}f(y)g(x-W_{t/\eps^2})q_{u_1}(y)q_{u_2}(x)dxdydu_1du_2,
\end{equation*}
and the rest is the same. The proof is complete.
\end{proof}

\begin{lemma}
Assume $f,g$ are bounded, integrable, positive, and radially symmetric and decreasing, then we have
\begin{equation}
\E_B\E_W\int_{[0,t/\eps^2]^3}f(x_{s_1}-B_{s_2})g(y_{s_3}-\tilde{W})ds \les \frac{1}{\eps^2}\sup_{y\in \R^d}\int_{\R^d}\frac{f(x+y)}{|x|^{d-2}}dx\sup_{y\in \R^d}\int_{\R^d}\frac{g(x+y)}{|x|^{d-2}}dx,
\end{equation}
where $x,y\in \{B,W\}, \tilde{W}\in \{0,B_{t/\eps^2},W_{t/\eps^2}\}$.
\label{lem:Brown3rdMo}
\end{lemma}

\begin{proof}
The proof is similar to that of Lemma \ref{lem:Brown2ndMo}. We do not present the details here.
\end{proof}

\begin{lemma}
Assume $f,g$ are bounded, integrable, positive, and radially symmetric and decreasing, then we have
\begin{equation}
\E_B\E_W\int_{[0,t/\eps^2]^4}f(x_{s_1}-B_{u_1})g(x_{s_2}-B_{u_2})dsdu \les \frac{1}{\eps^4}\sup_{y\in \R^d}\int_{\R^d}\frac{f(x+y)}{|x|^{d-2}}dx\sup_{y\in \R^d}\int_{\R^d}\frac{g(x+y)}{|x|^{d-2}}dx,
\end{equation}
where $x\in \{B,W\}$.
\label{lem:Brown4thMo}
\end{lemma}

\begin{proof}
The proof is similar to that of Lemma \ref{lem:Brown2ndMo}. We do not present the details here.
\end{proof}

\begin{lemma}
When $d=4$ and $s>0$,
\begin{equation}
\int_{\R^d}\frac{e^{-|\xi|^2s}}{(1+\frac12|\xi|^2)^2}d\xi\les1+|\log s|.
\end{equation}
\label{lem:proofLem711}
\end{lemma}

\begin{proof}
By a change of coordinate, 
\begin{equation}
\begin{aligned}
\int_{\R^d}\frac{e^{-|\xi|^2s}}{(1+\frac12|\xi|^2)^2}d\xi\les &\int_0^\infty \frac{e^{-r^2s}r^3}{(1+r^2)^2}dr\les 1+\int_1^\infty \frac{e^{-r^2s}}{r}dr\les 1+|\log s|.
\end{aligned}
\end{equation}
\end{proof}

\begin{lemma}
When $d=4$, $|\log\lambda|^{-1}\langle\Phi_\lambda,\Phi_\lambda\rangle\to 2(2\pi)^{-d}\hat{R}(0)$.
\label{lem:conVard4}
\end{lemma}

\begin{proof}
First we consider the following integral
\begin{equation}
\frac{1}{|\log\lambda|}\int_0^{\frac{1}{\sqrt{\lambda}}}\frac{f(\sqrt{\lambda}r)r^3}{(1+\frac12r^2)^2}dr
\end{equation}
for some smooth and fast-decaying $f$. By an integration by parts, we have
\begin{equation}
\frac{1}{|\log\lambda|}\int_0^{\frac{1}{\sqrt{\lambda}}}\frac{f(\sqrt{\lambda}r)r^3}{(1+\frac12r^2)^2}dr\to 2f(0)
\label{eq:proofLemconVar}
\end{equation}
as $\lambda\to 0$.

Secondly, we have
\begin{equation}
\begin{aligned}
\frac{\langle\Phi_\lambda,\Phi_\lambda\rangle}{|\log\lambda|}=&\frac{1}{(2\pi)^d|\log\lambda|}\int_{\R^d}\frac{\hat{R}(\xi)}{(\lambda+\frac12|\xi|^2)^2}d\xi\\
=&\frac{1}{(2\pi)^d|\log\lambda|}\left(\int_{|\xi|>1}+\int_{|\xi|\leq 1}\right)\frac{\hat{R}(\xi)}{(\lambda+\frac12|\xi|^2)^2}d\xi.
\end{aligned}
\end{equation}
Clearly the first part goes to zero. For the second part, by a change of variables, we obtain
\begin{equation}
\frac{1}{(2\pi)^d|\log\lambda|}\int_{|\xi|\leq 1}\frac{\hat{R}(\xi)}{(\lambda+\frac12|\xi|^2)^2}d\xi=
\frac{1}{(2\pi)^d|\log\lambda|}\int_{|\xi|\leq 1/\sqrt{\lambda}}\frac{\hat{R}(\sqrt{\lambda}\xi)}{(1+\frac12|\xi|^2)^2}d\xi.
\end{equation}
Since $R(x)$ decays sufficiently fast, $\hat{R}(\xi)$ is smooth, so by \eqref{eq:proofLemconVar}, the proof is complete.
\end{proof}

\section{Proof of Lemma \ref{lem:m1m3A1A1}, \ref{lem:sumG4A2A2}, \ref{lem:sumL4A2A2}, \ref{lem:sumG3A1A2}, \ref{lem:sumL2A1A2}}

\begin{proof}[Proof of Lemma \ref{lem:m1m3A1A1}]
For the indexes satisfying $m_1+m_2\geq 1, m_3+m_4\geq 1$, there are the following four cases.
\begin{enumerate}
\item $m_1m_3\neq 0$.
\item $m_2m_4\neq 0$.
\item $m_2=m_3=0$.
\item $m_1=m_4=0$.
\end{enumerate}
If $m_1m_3\neq 0$, we first consider the expectation in $W$. For any permutation of $\{u_1,\ldots,u_{m_2+m_4},\tilde{u}\}$, denoted by $\mathcal{S}$, we have
\begin{equation}
\begin{aligned}
&\int_{\mathcal{S}}\E_W\{\prod_{i=1}^{m_2}|\varphi|(W_{u_i}-y)\prod_{i=m_2+1}^{m_2+m_4}|\varphi|(W_{u_i}-z)|g_2|(W_{\tilde{u}}-z)\}dud\tilde{u}\\
=& \int_{0\leq u_1\ldots\leq u_{m_2+m_4+1}\leq t/\eps^2}\prod_{i=1}^{m_2+m_4+1}\E_W\{|\varphi_i|(W_{u_i}-y_i)\}du
\end{aligned}
\end{equation}
for $\varphi_i$ chosen as $\varphi, g_2$, and $y_i$ chosen as $y,z$ depending on the permutation. By a standard change of variables, we have
\begin{equation}
\begin{aligned}
&\int_{0\leq u_1\ldots\leq u_{m_2+m_4+1}\leq t/\eps^2}\prod_{i=1}^{m_2+m_4+1}\E_W\{|\varphi_i|(W_{u_i}-y_i)\}du\\
=&\int_{\R_+^{m_2+m_4+1}}\int_{\R^{(m_2+m_4+1)d}}1_{\sum_{i=1}^{m_2+m_4+1}u_i\leq \frac{t}{\eps^2}}\prod_{i=1}^{m_2+m_4+1}|\varphi_i|(\sum_{j=1}^ix_j-y_i)q_{u_i}(x_i)dxdu\\
\les & \int_{\R_+^{m_2+m_4+1}}\int_{\R^{(m_2+m_4+1)d}}1_{\sum_{i=1}^{m_2+m_4+1}\frac{|x_i|^2}{2\lambda_i}\leq \frac{t}{\eps^2}}\prod_{i=1}^{m_2+m_4+1}|\varphi_i|(\sum_{j=1}^ix_j-y_i)\frac{1}{|x_i|^{d-2}}\lambda_i^{\frac{d}{2}-2}e^{-\lambda_i}dxd\lambda\\
\les & \int_{\R_+^{m_2+m_4+1}}\int_{\R^{(m_2+m_4+1)d}}\prod_{i=1}^{m_2+m_4+1}|\varphi_i|(\sum_{j=1}^ix_j-y_i)\frac{1}{|x_i|^{d-2}}\lambda_i^{\frac{d}{2}-2}e^{-\lambda_i}dxd\lambda.
\end{aligned}
\end{equation}
First integrate in $\lambda_i, i=1,\ldots, m_2+m_4+1$, then in $x_{m_2+m_4+1},\ldots, x_1$, since $\int_{\R^d}|\varphi|(x+y)|x|^{2-d}dx$ is uniformly bounded in $y$, we have
\begin{equation}
\int_{\mathcal{S}}\E_W\{\prod_{i=1}^{m_2}|\varphi|(W_{u_i}-y)\prod_{i=m_2+1}^{m_2+m_4}|\varphi|(W_{u_i}-z)|g_2|(W_{\tilde{u}}-z)\}dud\tilde{u}\les \sup_{y\in \R^d}\int_{\R^d}\frac{|g_2|(x+y)|}{|x|^{d-2}}dx,
\end{equation}
which leads to
\begin{equation}
\begin{aligned}
&\int_{[0,t/\eps^2]^{m_2+m_4+1}}\E_W\{\prod_{i=1}^{m_2}|\varphi|(W_{u_i}-y)\prod_{i=m_2+1}^{m_2+m_4}|\varphi|(W_{u_i}-z)|g_2|(W_{\tilde{u}}-z)\}dud\tilde{u}\\
\leq &C^{m_2+m_4+1}(m_2+m_4+1)!\sup_{y\in \R^d}\int_{\R^d}\frac{|g_2|(x+y)|}{|x|^{d-2}}dx
\end{aligned}
\label{eq:m1m3EW}
\end{equation}
for some constant $C$.

Next, we consider the expectation in $B$. The analyze is similar except that we have to deal with integration in $y,z$. Again for any permutation of $\{s_1,\ldots,s_{m_1+m_3},\tilde{s}\}$, denoted by $\mathcal{S}$,  we consider
\begin{equation}
\begin{aligned}
&\int_{\R^{2d}}\int_{\mathcal{S}}\prod_{i=1}^{m_1}|\varphi|(B_{s_i}-y)\prod_{i=m_1+1}^{m_1+m_3}|\varphi|(B_{s_i}-z)|g_1|(B_{\tilde{s}}-y)dsd\tilde{s}dydz\\
=&
\int_{\R^{2d}}\int_{0\leq s_1\ldots\leq s_{m_1+m_3+1}\leq t/\eps^2}\prod_{i=1}^{m_1+m_3+1}\E_B\{|\varphi_i|(B_{s_i}-y_i)\}dsdydz\\
=&\int_{\R^{2d}}\int_{\R_+^{m_1+m_3+1}}\int_{\R^{(m_1+m_3+1)d}}1_{\sum_{i=1}^{m_1+m_3+1}u_i\leq \frac{t}{\eps^2}}\prod_{i=1}^{m_1+m_3+1}|\varphi_i|(\sum_{j=1}^ix_j-y_i)q_{u_i}(x_i)dxdudydz,
\end{aligned}
\end{equation}
 where $\varphi_i$ is either $\varphi$ or $g_1$ and $y_i$ is either $y$ or $z$ depending on the permutation. Let $i_y,i_z$ be the smallest indexes such that $y_{i_y}=y$ and $y_{i_z}=z$. By the same change of variables $ \lambda_i=\frac{|x_i|^2}{2u_i}$ for $i\neq i_y,i_z$, we have
 \begin{equation}
 \begin{aligned}
 &\int_{\R^{2d}}\int_{\mathcal{S}}\prod_{i=1}^{m_1}|\varphi|(B_{s_i}-y)\prod_{i=m_1+1}^{m_1+m_3}|\varphi|(B_{s_i}-z)|g_1|(B_{\tilde{s}}-y)dsd\tilde{s}dydz\\
 \les &
 \int_{\R^{2d}}\int_{\R_+^{m_1+m_3+1}}\int_{\R^{(m_1+m_3+1)d}}1_{u_{i_y}+u_{i_z}\leq \frac{t}{\eps^2}}\left(\prod_{i=1}^{m_1+m_3+1}|\varphi_i|(\sum_{j=1}^ix_j-y_i)\right)\\
 &\left(\prod_{i\neq i_y,i_z}\frac{1}{|x_i|^{d-2}}\lambda_i^{\frac{d}{2}-2}e^{-\lambda_i}\right)q_{u_{i_y}}(x_{i_y})q_{u_{i_z}}(x_{i_z})dxdud\lambda dydz.
\end{aligned}
\label{eq:dydzPoisson}
\end{equation}

Let $\tilde{i}_y$ be the second smallest index such that $y_{\tilde{i}_y}=y$. The following is the order in which we integrate with respect to $x,u,\lambda,y,z$ in \eqref{eq:dydzPoisson}. It ensures that the integral of $|g_1|$ always contains a factor of $1/|x|^{d-2}$.

First integrate in $\lambda_i$, then integrate in $x_{m_1+m+3+1},\ldots,x_{\max(\tilde{i}_y,i_z)+1}$.

If $i_z>\tilde{i}_y$, for $|\varphi_{i_z}|(\sum_{j=1}^{i_z}x_j-z)$, we integrate in $z$; next, we integrate in $x_{i_z},\ldots,x_{\tilde{i}_y+1}$. Since $i_z>\tilde{i}_y$, we have $i_y=1,\tilde{i}_y=2$. So we are left with $|\varphi_1|(x_1-y)|\varphi_2|(x_1+x_2-y)|x_2|^{2-d}q_{u_1}(x_1)$, integrate in $y,x_2,x_1$. In the end, we integrate in $u_{i_y},u_{i_z}$.

If $i_z<\tilde{i}_y$, for $|\varphi_{i_y}|(\sum_{j=1}^{i_y}x_j-y)|\varphi_{\tilde{i}_y}|(\sum_{j=1}^{\tilde{i}_y}x_j-y)$, integrate in $y, x_{\tilde{i}_y}$, then integrate in $x_{\tilde{i}_y-1},\ldots, x_{i_z+1}$; for $|\varphi_{i_z}|(\sum_{j=1}^{i_z}x_j-z)$, integrate in $z$. Since $i_z=1$ or $2$, we integrate in $x_1$, and in the end, integrate in $u_{i_y},u_{i_z}$.

After the above integration, and using the fact that $\int_{\R^d}|\varphi|(x+y)|x|^{2-d}dx$ is uniformly bounded in $y$, we arrive at the following estimate
\begin{equation}
\begin{aligned}
&\E_B\int_{\R^{2d}}\int_{\mathcal{S}}\prod_{i=1}^{m_1}|\varphi|(B_{s_i}-y)\prod_{i=m_1+1}^{m_1+m_3}|\varphi|(B_{s_i}-z)|g_1|(B_{\tilde{s}}-y)dsd\tilde{s}dydz\\
 \les & \frac{1}{\eps^4}\max(\sup_y \int_{\R^d}\frac{|g_1|(x+y)}{|x|^{d-2}}dx,\sup_y \int_{\R^d}\frac{(|g_1|\star |\varphi|)(x+y)}{|x|^{d-2}}dx),
\end{aligned}
\end{equation}
where the factor of $\eps^{-4}$ comes from integration in $u_{i_y},u_{i_z}$.

Therefore,
\begin{equation}
\begin{aligned}
&\E_B\int_{\R^{2d}}\int_{[0,t/\eps^2]^{m_1+m_3+1}}\prod_{i=1}^{m_1}|\varphi|(B_{s_i}-y)\prod_{i=m_1+1}^{m_1+m_3}|\varphi|(B_{s_i}-z)|g_1|(B_{\tilde{s}}-y)dsd\tilde{s}dydz\\
\leq &C^{m_1+m_3+1}(m_1+m_3+1)!\frac{1}{\eps^4}\max(\sup_y \int_{\R^d}\frac{|g_1|(x+y)}{|x|^{d-2}}dx,\sup_y \int_{\R^d}\frac{(|g_1|\star |\varphi|)(x+y)}{|x|^{d-2}}dx)
\end{aligned}
\label{eq:m1m3EB}
\end{equation}
for some constant $C$.

Now we only need to combine \eqref{eq:m1m3EW} and \eqref{eq:m1m3EB} together with Lemma \ref{lem:boundConvolution} to complete the proof for the case $m_1m_3\neq 0$.

If $m_2m_4\neq 0$, the discussion is the same by symmetry.

If $m_2=m_3=0$, the discussion is similar except that when taking $\E_W,\E_B$, we have to deal with the integral in $z,y$ respectively.

In the end, we deal with the case when $m_1=m_4=0$, so $m_2\geq 1, m_3\geq 1$.

We first look at the case when $m_2=m_3=1$, i.e., \begin{equation}
\begin{aligned}
&\E_B\E_W\int_{\R^{2d}}\int_{[0,t/\eps^2]^4}|\varphi|(B_s-z)|g_1|(B_{\tilde{s}}-y)|\varphi|(W_u-y)|g_2|(W_{\tilde{u}}-z)dsdud\tilde{s}d\tilde{u}dydz\\
=&\E_B\E_W \int_{[0,t/\eps^2]^4}(|g_1|\star|\varphi|)(B_{\tilde{s}}-W_u)(|g_2|\star|\varphi|)(W_{\tilde{u}}-B_s)dsdud\tilde{s}d\tilde{u}\\
\les &\frac{1}{\eps^4}\sup_y \int_{\R^d}\frac{(|g_1|\star |\varphi|)(x+y)}{|x|^{d-2}}dx\sup_y \int_{\R^d}\frac{(|g_2|\star |\varphi|)(x+y)}{|x|^{d-2}}dx\\
\les &\frac{|\log\eps|^2}{\eps}
\end{aligned}
\end{equation}
by Lemma \ref{lem:Brown4thMo} and \ref{lem:boundConvolution}.

Next, we look at the case when $m_2+m_3\geq 3$. By symmetry, we assume $m_2\geq 2$. Consider $\E_B$ and $dz$, by similar discussion as before, we obtain that
\begin{equation}
\begin{aligned}
&\E_B\int_{\R^d}\int_{[0,t/\eps^2]^{m_3+2}}\prod_{i=1}^{m_3}|\varphi|(B_{s_i}-z)|g_1|(B_{\tilde{s}}-y)|g_2|(W_{\tilde{u}}-z)dsd\tilde{s}d\tilde{u}dz\\
\les &(m_3+1)!\frac{1}{\eps^2}\sup_y \int_{\R^d}\frac{(|g_2|\star |\varphi|)(x+y)}{|x|^{d-2}}dx\sup_y \int_{\R^d}\frac{|g_1|(x+y)}{|x|^{d-2}}dx.
\end{aligned}
\end{equation}
Consider $\E_W$ and $dy$, we obtain that
\begin{equation}
\begin{aligned}
\E_W\int_{\R^d}\int_{[0,t/\eps^2]^{m_2}}\prod_{i=1}^{m_2}|\varphi|(W_{u_i}-y)dudy\les m_2!\frac{1}{\eps^2}.
\end{aligned}
\end{equation}
Combining them with Lemma \ref{lem:boundConvolution}, the proof is complete.
\end{proof}

\begin{proof}[Proof of Lemma \ref{lem:sumG4A2A2}]

First, we note that $f^\lambda$ is uniformly bounded, since $\F\{f^\lambda\}(\xi)=\F\{\varphi\}(\xi)(\lambda+\frac12|\xi|^2)^{-1}$ is bounded in $L^1$.

Similarly, there are four cases.

If $m_1m_3\neq 0$, we use a constant to bound $f^\lambda$, and the rest of the discussion is similar to the proof of Lemma \ref{lem:m1m3A1A1}; i.e., first take $\E_W$, then take $\E_B$ while dealing with integrals in $y,z$. We get the following estimate
\begin{equation}
\begin{aligned}
&\E_B\E_W\int_{\R^{2d}}\int_{[0,t/\eps^2]^{N(m_i)}}\prod_{i=1}^{m_1}|\varphi|(B_{s_i}-y)\prod_{i=m_1+1}^{m_1+m_3}|\varphi|(B_{s_i}-z)|f^\lambda|(\tilde{B}-y)\\
&\quad\quad\quad\quad\prod_{i=1}^{m_2}|\varphi|(W_{u_i}-y)\prod_{i=m_2+1}^{m_2+m_4}|\varphi|(W_{u_i}-z)|f^\lambda|(\tilde{W}-z)dsdudydz\\
\les & (m_1+m_3)!(m_2+m_4)!\frac{1}{\eps^4}
\end{aligned}
\label{eq:m1m3A2A2bound1}
\end{equation}

If $m_2m_4\neq 0$, by symmetry, we get the same estimate as in \eqref{eq:m1m3A2A2bound1}.

If $m_2=m_3=0$ or $m_1=m_4=0$, again we bound $|f^\lambda|$ by constant, and when taking $\E_W,\E_B$, deal with the integral in $z,y$ respectively. In the end, we get the same estimate as in \eqref{eq:m1m3A2A2bound1}, which completes the proof.
\end{proof}

\begin{proof}[Proof of Lemma \ref{lem:sumL4A2A2}]

If $\sum_{i=1}^4m_i=2$, we have
\begin{equation}
\begin{aligned}
&\int_{\R^{2d}}g_B(y)^{m_1}g_W(y)^{m_2}g_B(z)^{m_3}g_W(z)^{m_4}h_B(y)\overline{h_W}(z)dydz\\
=&\eps^3\int_{[0,t/\eps^2]^2} F^\lambda(x_s-\tilde{B})F^\lambda(y_u-\tilde{W})dsdu,
\end{aligned}
\end{equation}
where $F^\lambda(x)=\int_{\R^d}\varphi(x+y)f^\lambda(y)dy$, and $x,y\in \{B,W\}$, $\tilde{B}\in \{0,B_{t/\eps^2}\}$, $\tilde{W}\in \{0,W_{t/\eps^2}\}$. Note that $|F^\lambda|(x)\leq |f^\lambda|\star|\varphi|(-x)$ since $|\varphi|$ is symmetric. In the following, we will always replace $|f^\lambda|$ by $|\varphi|\star(e^{-c\sqrt{\lambda}|x|}|x|^{2-d})$ in the estimates, so we can assume it is radially symmetric. By Lemma \ref{lem:Brown2ndMo}, we have that
\begin{equation}
\begin{aligned}
\E_B\E_W|\eps^3\int_{[0,t/\eps^2]^2} F^\lambda(x_s-\tilde{B})F^\lambda(y_u-\tilde{W})dsdu|\les& \eps^3\left(\sup_{y\in \R^d}\int_{\R^d}\frac{(|f^\lambda|\star|\varphi|)(x+y)}{|x|^{d-2}}dx\right)^2\\
\les &\eps,
\end{aligned}
\end{equation}
where the last inequality comes from Lemma \ref{lem:boundConvolution}. 

If $\sum_{i=1}^4m_i=3$ and $m_i\neq 2$ for all $i$, without loss of generality assume $m_1=0$, so we have
\begin{equation}
\begin{aligned}
&\E_B\E_W\eps^4\int_{\R^{2d}}\int_{[0,t/\eps^2]^3}|\varphi|(B_{s_1}-z)|f^\lambda|(\tilde{B}-y)|\varphi|(W_{u_1}-y)|\varphi|(W_{u_2}-z)|f^\lambda|(\tilde{W}-z)dsdudydz\\
\les&\E_B\E_W\eps^4 \int_{[0,t/\eps^2]^3}(|\varphi|\star|\varphi|)(B_{s_1}-W_{u_2})(|f^\lambda|\star|\varphi|)(\tilde{B}-W_{u_1})dsdu\\
\les & \eps^4\frac{1}{\eps^2}\sup_{y\in \R^d}\int_{\R^d}\frac{(|f^\lambda|\star|\varphi|)(x+y)}{|x|^{d-2}}dx\les \eps,
\end{aligned}
\end{equation}
where in the first inequality, we bound $|f^\lambda|(\tilde{W}-z)$ by a constant, while in the second and third inequalities, we apply Lemma \ref{lem:Brown3rdMo} and \ref{lem:boundConvolution}.

If $\sum_{i=1}^4m_i=3$ and $m_i=2$ for some $i$, there are two cases by symmetry, $m_1=1,m_3=2$ or $m_1=1,m_4=2$.

When $m_1=1,m_3=2$, similarly we have
\begin{equation}
\begin{aligned}
&\E_B\E_W\eps^4\int_{\R^{2d}}\int_{[0,t/\eps^2]^3}|\varphi|(B_{s_1}-y)|\varphi|(B_{s_2}-z)|\varphi|(B_{s_3}-z)|f^\lambda|(\tilde{B}-y)|f^\lambda|(\tilde{W}-z)dsdydz\\
\les &\E_B\eps^4 \int_{[0,t/\eps^2]^3}(|\varphi|\star|\varphi|)(B_{s_2}-B_{s_3})(|f^\lambda|\star|\varphi|)(\tilde{B}-B_{s_1})ds\\
\les & \eps^4\frac{1}{\eps^2}\sup_{y\in \R^d}\int_{\R^d}\frac{(|f^\lambda|\star|\varphi|)(x+y)}{|x|^{d-2}}dx\les \eps
\end{aligned}
\end{equation}
by Lemma \ref{lem:Brown3rdMo}.

When $m_1=1,m_4=2$, similarly we have
\begin{equation}
\begin{aligned}
&\E_B\E_W\eps^4\int_{\R^{2d}}\int_{[0,t/\eps^2]^3}|\varphi|(B_{s_1}-y)|\varphi|(W_{u_1}-z)|\varphi|(W_{u_2}-z)|f^\lambda|(\tilde{B}-y)|f^\lambda|(\tilde{W}-z)dsdudydz\\
\les &\E_B\E_W\eps^4 \int_{[0,t/\eps^2]^3}(|\varphi|\star|\varphi|)(W_{u_1}-W_{u_2})(|f^\lambda|\star|\varphi|)(\tilde{B}-B_{s_1})dsdu\\
\les & \eps^4\frac{1}{\eps^2}\sup_{y\in \R^d}\int_{\R^d}\frac{(|f^\lambda|\star|\varphi|)(x+y)}{|x|^{d-2}}dx\les \eps
\end{aligned}
\end{equation}
by Lemma \ref{lem:Brown3rdMo}. The proof is complete.
\end{proof}

\begin{proof}[Proof of Lemma \ref{lem:sumG3A1A2}]

The discussion is similar as in the proof of Lemma \ref{lem:m1m3A1A1}, so we do not present all the details.

If $m_1m_3\neq 0$, we first use constant to bound $|f^\lambda|$, then take $\E_W$. Next we take $\E_B$ and deal with the integral in $y,z$. In the end, we obtain
\begin{equation}
\begin{aligned}
&\eps^{N(m_i)+\frac12}\E_B\E_W\int_{\R^{2d}}\int_{[0,t/\eps^2]^{N(m_i)+1}}\prod_{i=1}^{m_1}|\varphi|(B_{s_i}-y)\prod_{i=m_1+1}^{m_1+m_3}|\varphi|(B_{s_i}-z)|g|(B_s-y)\\
&\quad\quad\quad\quad\prod_{i=1}^{m_2}|\varphi|(W_{u_i}-y)\prod_{i=m_2+1}^{m_2+m_4}|\varphi|(W_{u_i}-z)|f^\lambda|(\tilde{W}-z)dsdudydz\\
\les & \eps^{N(m_i)+\frac12-4}(m_2+m_4)!(m_1+m_3+1)!\sup_{y\in \R^d}\int_{\R^d}\frac{(|\varphi|\star|g|)(x+y)}{|x|^{d-2}}dx\\
\les & \eps^{N(m_i)-2}|\log\eps|(m_2+m_4)!(m_1+m_3+1)!
\end{aligned}
\end{equation}

For other cases, the discussion is similar. The proof is complete.
\end{proof}

\begin{proof}[Proof of Lemma \ref{lem:sumL2A1A2}]
When $\sum_{i=1}^4m_i=2$, we have
\begin{equation}
\begin{aligned}
&|\int_{\R^{2d}}g_B(y)^{m_1}g_W(y)^{m_2}g_B(z)^{m_3}g_W(z)^{m_4}h_B(y)\overline{h_W}(z)dydz|\\
\leq &\eps^{2+\frac12}\int_{[0,t/\eps^2]^3}R_1(x_{s_1}-B_{s_2})R_2(y_{s_3}-\tilde{W})ds,
\end{aligned}
\end{equation}
where $R_1(x)=\int_{\R^d}|\varphi|(x+y)|g|(y)dy$, $R_2(x)=\int_{\R^d}|\varphi|(x+y)|f^\lambda|(y)dy$, and $x,y\in \{B,W\}, \tilde{W}\in \{0,W_{t/\eps^2}\}$. If we replace $|g|$ and $|f^\lambda|$ by the corresponding radially symmetric and decreasing bound, then by Lemma \ref{lem:Brown3rdMo}, we have
\begin{equation}
\begin{aligned}
&\E_B\E_W\int_{[0,t/\eps^2]^3}R_1(x_{s_1}-B_{s_2})R_2(y_{s_3}-\tilde{W})ds\\
\leq& \frac{1}{\eps^2}\sup_{y\in \R^d}\int_{\R^d}\frac{R_1(x+y)}{|x|^{d-2}}dx\sup_{y\in \R^d}\int_{\R^d}\frac{R_2(x+y)}{|x|^{d-2}}dx,
\end{aligned}
\end{equation}
so by Lemma \ref{lem:boundConvolution}, we finally obtain
\begin{equation}
\begin{aligned}
&\E_B\E_W|\int_{\R^{2d}}g_B(y)^{m_1}g_W(y)^{m_2}g_B(z)^{m_3}g_W(z)^{m_4}h_B(y)\overline{h_W}(z)dydz|\\
\leq &\eps^{2+\frac12}\frac{1}{\eps^2}\eps^{\frac32}|\log\eps|\frac{1}{\eps}=\eps|\log\eps|,
\end{aligned}
\end{equation}
which completes the proof.
\end{proof}

\section{On \eqref{eq:useMalliavin}: duality relation in Malliavin calculus}
\label{sec:malliavin}

Let $H=\oplus^d L^2([0,t])$ and $B_t$ a standard Brownian motion in $\R^d$. Take the isonormal Gaussian space $\{W(h)\}$ on $H$ defined as $W(h)=\sum_{k=1}^d \int_0^t \tilde{h}_k(s)dB_s^k$ when $h=(\tilde{h}_1,\tilde{h}_2,\ldots, \tilde{h}_d)\in H$,  then $B_t$ is written as
$$B_t=(W(h_1), W(h_2),\ldots, W(h_d))$$ with $h_i\in \oplus^d L^2([0,t])$, and only its $i-$ component is non-zero and equal to $1_{[0,t]}$.

Let $F=f(W(h_1),W(h_2), \ldots, W(h_d))$ for any test function $f$, and $G=(g_1(B_s),g_2(B_s),\ldots,g_d(B_s))$. Then the Skorohod integral for $G$ is defined as
\begin{equation}
\delta(G)=\sum_{k=1}^d \int_0^t g_k(B_s)dB_s^k.
\end{equation}

The duality relation reads
\begin{equation}
\E\{ F\delta(G)\}=\E\{\langle DF, G\rangle_H\},
\label{eq:duality}
\end{equation}
with $D$ the Malliavin derivative operator.

We have $DF=\sum_{k=1}^d \partial_k f(B_t)h_k$, so
\begin{equation}
\langle DF, G\rangle_H=\sum_{k=1}^d \partial_k f(B_t)\int_0^t g_k(B_s)ds.
\end{equation}
Therefore, \eqref{eq:duality} implies
\begin{equation}
\sum_{k=1}^d \E\{ f(B_t) \int_0^t g_k(B_s)dB_s^k\}=\sum_{k=1}^d \E\{ \partial_k f(B_t)\int_0^t g_k(B_s)ds\}.
\end{equation}

\def\cprime{$'$}


\end{document}